\documentclass[12pt,a4paper]{smfart}

\usepackage[T1]{fontenc}
\usepackage{lmodern}
\usepackage[headings]{fullpage}
\usepackage{amssymb}
\usepackage{mathrsfs}
\usepackage[frenchb]{babel}

\theoremstyle{plain}
\newtheorem{defi}{\definame}[section]
\newtheorem{prop}[defi]{\propname}
\newtheorem{theo}[defi]{\theoname}

\newtheorem{coro}[defi]{\coroname}

\newtheorem{lemm}[defi]{\lemmname}
\theoremstyle{definition}
\theoremstyle{remark}
\newtheorem{rema}[defi]{\remaname}
\newtheorem{exem}[defi]{\exemname}

\let\mathcal\mathscr

\newcommand{\Qp}{\mathbf{Q}_p}
\newcommand{\Cp}{\mathbf{C}_p}
\newcommand{\Zp}{\mathbf{Z}_p}
\newcommand{\Fp}{\mathbf{F}_p}
\newcommand{\NN}{\mathbf{N}}
\newcommand{\ZZ}{\mathbf{Z}}
\newcommand{\OO}{\mathcal{O}}
\newcommand{\MM}{\mathfrak{m}}
\newcommand{\Qpbar}{\overline{\mathbf{Q}}_p}
\renewcommand{\geq}{\geqslant}
\renewcommand{\leq}{\leqslant} 
\newcommand{\Gal}{\operatorname{Gal}}
\newcommand{\Mat}{\operatorname{Mat}}
\newcommand{\Lie}{\operatorname{Lie}}
\newcommand{\Sym}{\operatorname{Sym}}

\newcommand{\GL}{\mathrm{GL}}
\newcommand{\SL}{\mathrm{SL}}
\newcommand{\LieSL}{\mathfrak{sl}_2}
\newcommand{\Id}{\mathrm{Id}}
\newcommand{\Nm}{\mathrm{N}}
\newcommand{\Tr}{\mathrm{Tr}}

\newcommand{\vp}{\mathrm{val}_p}

\newcommand{\cyc}{\mathrm{cyc}}
\newcommand{\atplus}{\tilde{\mathbf{A}}^+}
\newcommand{\bt}{\widetilde{\mathbf{B}}}
\newcommand{\btplus}{\widetilde{\mathbf{B}}^+}
\newcommand{\bcris}{\mathbf{B}_{\mathrm{cris}}}  
\newcommand{\bdr}{\mathbf{B}_{\mathrm{dR}}}  
\newcommand{\bsen}{\mathbf{B}_{\mathrm{Sen}}}  
\newcommand{\asenM}{\mathbf{A}_{\mathrm{Sen},M}}  
\newcommand{\bsenM}{\mathbf{B}_{\mathrm{Sen},M}}  
\newcommand{\asen}{\mathbf{A}_{\mathrm{Sen}}}  
\newcommand{\dsen}{\mathrm{D}_{\mathrm{Sen}}}  
\newcommand{\dcroc}[1]{[\![ #1 ]\!]}
\newcommand{\dacc}[1]{\{\!\{ #1 \}\!\}}
\newcommand{\LT}{\mathrm{LT}}
\newcommand{\fin}{\mathrm{fin}}
\newcommand{\emb}{\mathrm{E}}

\newcommand{\an}{\mathrm{an}}
\newcommand{\la}{\mathrm{la}}
\newcommand{\dan}{\text{$\mbox{-}\mathrm{an}$}}
\newcommand{\fan}{\text{$F\!\mbox{-}\mathrm{la}$}}
\newcommand{\smat}[1]{\left( \begin{smallmatrix} #1 \end{smallmatrix} \right)}
\newcommand{\cbf}{\mathbf{c}}
\newcommand{\ibf}{\mathbf{i}}
\newcommand{\jbf}{\mathbf{j}}
\newcommand{\kbf}{\mathbf{k}}
\newcommand{\xbf}{\mathbf{x}}
\newcommand{\Tbf}{\mathbf{T}}

\author{Laurent Berger}
\address{UMPA de l'ENS de Lyon \\
UMR 5669 du CNRS \\ IUF}
\email{laurent.berger@ens-lyon.fr}
\urladdr{perso.ens-lyon.fr/laurent.berger/}

\author{Pierre Colmez}
\address{UPMC \\
Institut de Math\'ematiques de Jussieu \\
UMR 7586 du CNRS}
\email{pierre.colmez@imj-prg.fr}
\urladdr{www.math.jussieu.fr/~colmez/}
 
\date{\today}

\title[Th\'eorie de Sen et vecteurs localement analytiques]{Th\'eorie de Sen et vecteurs localement analytiques}

\alttitle{Sen theory and locally analytic vectors}



\begin{document}

\begin{abstract}
Nous g\'en\'eralisons la th\'eorie de Sen \`a des extensions $K_\infty/K$ dont le groupe de Galois 
est un groupe de Lie $p$-adique de dimension quelconque. Pour cela, nous rempla\c{c}ons l'espace des vecteurs $K$-finis de Sen par celui des vecteurs localement analytiques de Schneider et Teitelbaum. On obtient alors un espace vectoriel sur le corps des vecteurs localement analytiques de $\hat{K}_\infty$. Nous 
d\'ecrivons ce corps en portant une attention particuli\`ere au cas d'une extension de Lubin-Tate.
\end{abstract}

\begin{altabstract}
We generalize Sen theory to extensions $K_\infty/K$ whose Galois group is a $p$-adic Lie group 
of arbitrary dimension. To do so, we replace Sen's space of $K$-finite vectors by Schneider and Teitelbaum's space of locally analytic vectors. One then gets a vector space over the field of locally analytic vectors of $\hat{K}_\infty$. We 
describe this field in general and pay a special attention to the case of Lubin-Tate extensions.
\end{altabstract}

\maketitle

\tableofcontents

\setlength{\baselineskip}{18pt}

\section{Introduction}\label{intro}
\subsection{Descente presque \'etale}
On fixe une cl\^oture alg\'ebrique $\Qpbar$ de $\Qp$ et on note $\Cp$ le compl\'et\'e de $\Qpbar$ pour la norme $p$-adique. 

Si $K$ est une extension finie de $\Qp$ contenue dans $\Qpbar$, et si $G_K=\Gal(\Qpbar/K)$,
 une id\'ee qui s'est av\'er\'ee fructueuse pour l'\'etude des repr\'esentations $p$-adiques de $G_K$, 
est de d\'evisser $\Qpbar$ en introduisant une extension interm\'ediaire $K \subset K_\infty \subset \Qpbar$,
 telle que $K_\infty/K$ ne soit pas trop compliqu\'ee, mais quand m\^eme profond\'ement ramifi\'ee 
(voir~\cite{CG}) de telle sorte que $\Qpbar/K_\infty$ soit presque \'etale au sens de Faltings. 
On note $H_K$ le groupe $\Gal(\Qpbar/K_\infty)$ et, si $K_\infty$ est une extension galoisienne de $K$, 
on note $\Gamma_K$ le groupe $\Gal(K_\infty/K)=G_K/H_K$. Le fait que l'extension $\Qpbar/K_\infty$ est 
presque \'etale a pour cons\'equence le r\'esultat de descente presque \'etale suivant (cf.~\cite{SN80}).

\begin{theo}\label{thS1}
Si $d \geq 1$, alors $H^1(H_K,\GL_d(\Cp)) =\{1\}$.
\end{theo}
\begin{rema}\label{rem1}
{\rm (i)} D'apr\`es le th\'eor\`eme d'Ax-Sen-Tate, $\Cp^{H_K}$ est l'adh\'erence $\hat{K}_\infty$ de $K_\infty$ dans $\Cp$ (autrement dit, c'est le compl\'et\'e de $K_\infty$ pour la norme $p$-adique). 
Le th.~\ref{thS1} se traduit par le fait que, si
$X$ est une $\Cp$-repr\'esentation semi-lin\'eaire de dimension finie $d$ de $G_K$
(par exemple si $X=\Cp\otimes_{\Qp}V$, o\`u $V$ est une $\Qp$-repr\'esentation lin\'eaire
de $G_K$, de dimension~$d$),
l'application $\Cp \otimes_{\hat{K}_\infty} X^{H_K} \to X$ est un isomorphisme.
En particulier, si $K_\infty/K$ est galoisienne, 
$W=X^{H_K}$ est une $\hat{K}_\infty$-repr\'esentation semi-lin\'eaire 
de dimension~$d$ de $\Gamma_K$.
On est donc naturellement amen\'e \`a \'etudier 
les $\hat{K}_\infty$-repr\'esentations semi-lin\'eaires
de dimension finie de $\Gamma_K$ pour associer 
des invariants aux $\Qp$-repr\'esentations lin\'eaires
de $G_K$; c'est exactement ce qu'a fait Sen~\cite{SN80} pour d\'efinir les poids
de Hodge-Tate d'une repr\'esentation quelconque
(cf.~rem.~\ref{rem2} ci-dessous).

{\rm (ii)} On a le m\^eme genre d'\'enonc\'e en rempla\c cant $\Cp$ par $\bdr^+$ \cite{F04}
ou par $\bt^\dagger$ \cite{CC98,LB11} ou encore par $\bcris^{\varphi=1}$ \cite{PC98,LB8,FF}
(et $\hat{K}_\infty$ par les points fixes par $H_K$ de l'anneau correspondant).
\end{rema}
On peut, par exemple, prendre pour $K_\infty$ une des extensions suivantes :

\quad --- l'extension cyclotomique $K(\mu_{p^\infty})$;

\quad --- l'extension de Kummer $K(\sqrt[p^\infty]{\pi})$, o\`u $\pi$ est une uniformisante de $K$;

\quad --- la compos\'ee $K(\mu_{p^\infty},\sqrt[p^\infty]{\pi})$ des deux extensions pr\'ec\'edentes;

\quad --- une extension de type Lubin-Tate, obtenue en rajoutant \`a $K$ les points de $p^\infty$-torsion d'un groupe de Lubin-Tate associ\'e \`a une uniformisante d'un sous-corps de $K$;

\quad --- une extension galoisienne infiniment ramifi\'ee de groupe de Galois 
un groupe de Lie $p$-adique (qui est alors profond\'ement ramifi\'ee, cf. \cite{S72} et \cite{CG}).

Chacun des exemples ci-dessus a son int\'er\^et propre :

\quad $\bullet$
L'extension la plus utilis\'ee est l'extension cyclotomique, par exemple dans la th\'eorie des $(\varphi,\Gamma)$-modules (cf.~\cite{F90} et \cite{CC98}), mais pour beaucoup de questions, il semble naturel d'utiliser d'autres extensions.

\quad $\bullet$
De beaucoup de points de vue, la plus simple serait l'extension de Kummer mais elle n'est pas galoisienne, ce qui pose de s\'erieux probl\`emes (elle s'est quand m\^eme r\'ev\'el\'ee tr\`es utile pour l'\'etude des repr\'esentations semi-stables~\cite{CBun,MK06,BTR13}) et on lui pr\'ef\`ere parfois~\cite{FTR,XC13} la compos\'ee $K(\mu_{p^\infty},\sqrt[p^\infty]{\pi})$.

\quad $\bullet$ 
En vue d'une extension de la correspondance de Langlands locale $p$-adique \`a $\GL_2(K)$, il semble naturel de consid\'erer une extension de Lubin-Tate associ\'ee \`a une uniformisante de $K$ \cite{KR09,FX13,PGMLAV} car cela rend la correspondance pour $\GL_1(K)$ compl\`etement transparente.

\quad $\bullet$ 
Enfin, pour des applications \`a la th\'eorie d'Iwasawa non commutative \cite{MH79,JC99,V03}, le cadre naturel
est celui d'une extension de groupe de Galois un groupe de Lie $p$-adique arbitraire (ce qui inclut tous
les cas pr\'ec\'edents \`a l'exception de l'extension de Kummer).

Malheureusement, si $\Gamma_K$ est de dimension~$\geq 2$, certains outils fondamentaux de la th\'eorie cyclotomique manquent \`a l'appel (comme l'existence des traces normalis\'ees continues sur $K_\infty$ de \cite{T67}),
ce qui rend la th\'eorie nettement plus d\'elicate, et explique qu'elle soit moins d\'evelopp\'ee.
Le but de cet article et de~\cite{PGMLAV} est de sugg\'erer que l'on peut remplacer 
ces outils manquants
par des \'el\'ements de la th\'eorie des repr\'esentations de~$\Gamma_K$ : 
le concept de vecteur localement analytique est utilis\'e dans cet article
pour \'etendre la th\'eorie de Sen; dans~\cite{PGMLAV}, ce concept permet
de d\'efinir des invariants palliant le manque de surconvergence des $(\varphi,\Gamma)$-modules dans le cas d'une extension de Lubin-Tate. 

Dans tout le reste de l'article, 
{\it on suppose que $K_\infty/K$ est galoisienne, que $\Gamma_K=\Gal(K_\infty/K)$ est un groupe
de Lie $p$-adique, et que le sous-groupe d'inertie de $\Gamma_K$
est infini} (c'est automatique si $\dim \Gamma_K\geq 2$).
\subsection{Vecteurs $K$-finis}
Soit $W$ une $\hat{K}_\infty$-repr\'esentation semi-lin\'eaire de dimension finie de $\Gamma_K$. Si $w \in W$, 
disons que $w$ est {\it $K$-fini} s'il appartient \`a un sous $K$-espace vectoriel de dimension finie de $W$ qui est 
stable par $\Gamma_K$. Soit $W^{\fin}$ l'ensemble des vecteurs $K$-finis de $W$. 
C'est un sous $K_\infty$-espace vectoriel de $W$. 

Si $\dim\Gamma_K =1$, on dispose du r\'esultat suivant de Sen (cf.\ \cite{SN80}). 

\begin{theo}\label{thS2}
On a $\hat{K}_\infty^{\fin}=K_\infty$ et, plus g\'en\'eralement,
si $W$ est une $\hat{K}_\infty$-repr\'esen\-tation semi-lin\'eaire de dimension finie de $\Gamma_K$, 
l'application $\hat{K}_\infty \otimes_{K_\infty} W^{\fin} \to W$ est un isomorphisme.
\end{theo}

\begin{rema}\label{rem2}
{\rm (i)} 
L'alg\`ebre de Lie de $\Gamma_K$ agit lin\'eairement sur le $K_\infty$-espace vectoriel $W^{\fin}$.
Cette alg\`ebre est de rang~$1$ sur $\Zp$ et, si $K_\infty/K$ est l'extension cyclotomique, elle admet
un g\'en\'erateur canonique $\nabla=\lim_{\gamma\to 1}\frac{\gamma-1}{\chi_{\rm cycl}(\gamma)-1}$.
 L'op\'erateur de $W^{\fin}$ ainsi d\'efini
est {\it l'op\'erateur de Sen} $\Theta_{\mathrm{Sen}}$. Ses valeurs propres sont les
{\it poids de Hodge-Tate} de $W$.

{\rm (ii)}
Si $W=(\Cp\otimes_{\Qp} V)^{H_K}$ o\`u $V$ est une $\Qp$-repr\'esentation de $G_K$ et $K_\infty/K$
est l'extension cyclotomique, 
on note $\dsen(V)$ l'espace $W^{\fin}$ et les {\it poids de Hodge-Tate de $V$} sont, par d\'efinition,
ceux de $W$.  De plus:

\quad $\bullet$ L'application naturelle $\Cp\otimes_{K_\infty}\dsen(V)\to \Cp \otimes_{\Qp} V$
est un isomorphisme de repr\'esentations $\Cp$-semi-lin\'eaires de $G_K$.

\quad $\bullet$ Si on \'etend $\Theta_{\rm Sen}$ par $\Cp$-lin\'earit\'e \`a
$\Cp\otimes_{K_\infty}\dsen(V)$, alors $\Theta_{\rm Sen}$ commute \`a l'action
de $G_K$ et on a
$$\big(\Cp\otimes_{K_\infty}\dsen(V)\big)^{\Theta_{\rm Sen}=0}=\Cp\otimes_K (\Cp\otimes_{\Qp} V)^{G_K}.$$

\quad $\bullet$ $\Theta_{\rm Sen}=0$ (ce qui \'equivaut \`a ce que $V$ soit de Hodge-Tate
\`a poids de Hode-Tate tous nuls) 
si et seulement si le sous-groupe d'inertie de $G_K$
agit \`a travers un quotient fini sur~$V$ (cf.~\S 5 de \cite{SHT}).

\quad $\bullet$ 
$\Theta_{\rm Sen}$ appartient au sous-$\Cp$-espace vectoriel de $\Cp\otimes_{\Qp} {\rm End}(V)$
engendr\'e par l'alg\`ebre de Lie ${\mathfrak g}$ de l'image de $G_K$ dans $\GL(V)$,
et ${\mathfrak g}$ est le plus petit sous-$\Qp$-espace vectoriel de
${\rm End}(V)$ ayant cette propri\'et\'e (cf.~\S 3.2 de \cite{SN80}).

{\rm (iii)}  L'espace $\dsen(V)$ du (ii) admet, si $n$ est assez grand,
un unique sous-$K_n$-espace vectoriel ${\rm D}_{{\rm Sen},n}(V)$ (avec $K_n=K(\mu_{p^n})$), stable par
$\Gamma_K$ et tel que l'application naturelle $K_\infty\otimes_{K_n}{\rm D}_{{\rm Sen},n}(V)\to \dsen(V)$
soit un isomorphisme (il en est alors de m\^eme
de l'application $\Cp\otimes_{K_n}{\rm D}_{{\rm Sen},n}(V)\to \Cp\otimes_{\Qp} V$);
ce sous-espace est stable par $\Theta_{\rm Sen}$. 
\end{rema}

Le fait que $W^{\fin}$ n'est pas un objet adapt\'e si $\dim\Gamma_K\geq 2$ avait \'et\'e 
observ\'e par Sen lui-m\^eme. A titre d'exemple, signalons le r\'esultat suivant (prop.~\ref{vecfinec}). Soit $\Gamma_K$ un sous-groupe ouvert de $\SL_2(\Zp)$ et $\pm s$ les deux poids de Hodge-Tate de la repr\'esentation d\'eduite de $G_K \to \Gamma_K \to \SL_2(\Zp)$.

\begin{prop}\label{thsl2}
Soit $\Gamma_K$ comme ci-dessus, avec $s \neq 0$, 
et soit $W$ une $\hat{K}_\infty$-repr\'esentation semi-lin\'eaire de dimension finie de $\Gamma_K$.

{\rm (i)}
 Si $W^{\fin} \neq \{0\}$, alors $W$ a un poids de Hodge-Tate qui appartient \`a $s \cdot \ZZ$;

{\rm (ii)}
Si $W^{\fin}$ contient une base de $W$, alors l'op\'erateur de Sen de $W$ est semisimple, \`a valeurs propres dans $s \cdot \ZZ$.
\end{prop}

\subsection{Vecteurs localement analytiques}
L'id\'ee principale de cet article est de remplacer $W^{\fin}$ par l'espace des vecteurs localement $\Qp$-analytiques $W^{\la}$ dont nous rappelons maintenant la d\'efinition. 

Soit $G$ un groupe de Lie $p$-adique (par exemple $\Gamma_K$), et soit $W$ un $\Qp$-espace de Banach qui est une repr\'esentation de $G$. Si $w \in W$, alors suivant le \S\,7 de \cite{ST03}, 
nous disons que $w$ est {\it localement $\Qp$-analytique} si l'application {\og orbite \fg} $G \to W$, 
donn\'ee par $g \mapsto g ( w)$, est une fonction localement $\Qp$-analytique sur $G$. 
On note $W^{\la}$ l'espace des vecteurs localement $\Qp$-analytiques de~$W$.
On a $W^{\fin}\subset W^{\la}$ d'apr\`es un analogue d'un r\'esultat classique de Cartan (\S\,V.9 de \cite{SLALG}). Notre premier r\'esultat (th.~\ref{finisan}), qui montre qu'en dimension~$1$ on retombe
sur les objets introduits par Sen, est le suivant.

\begin{theo}\label{thmA}
Si $\dim\Gamma_K=1$ et si $W$ est une $\hat{K}_\infty$-repr\'esen\-tation 
semi-lin\'eaire de dimension finie de $\Gamma_K$, alors $W^{\fin} = W^{\la}$.
\end{theo}

L'analogue du th.~\ref{thS2} dans le cas o\`u $\Gamma_K$ est de dimension quelconque
est le r\'esultat suivant (th.~\ref{anisok}) qui montre que $W^{\la}$ est un invariant fin de $W$.

\begin{theo}\label{thmB}
Si $\Gamma_K$ est un groupe de Lie $p$-adique, et si $W$ est une $\hat{K}_\infty$-repr\'esen\-tation semi-lin\'eaire de dimension finie de $\Gamma_K$, 
l'application naturelle $\hat{K}_\infty \otimes_{\hat{K}_\infty^{\la}} W^{\la} \to W$ est un isomorphisme.
\end{theo}

\begin{rema}\label{rem3}
{\rm (i)}  Le th.~\ref{thmB} soul\`eve la question de la description de $\hat{K}_\infty^{\la}$. 
On montre facilement (lem.~\ref{prodwan}) que $\hat{K}_\infty^{\la}$ est toujours un corps.
Si $\dim\Gamma_K=1$, alors $\hat{K}_\infty^{\la} = K_\infty$ par le th.~\ref{thmA}, 
mais si $\dim\Gamma_K\geq 2$, alors $\hat{K}_\infty^{\la}$ contient strictement $K_\infty$ 
(cf.~th.~\ref{thgen}).

{\rm (ii)} Le th.~\ref{thmB} fait \'echo au r\'esultat de Schneider et Teitelbaum~\cite{ST03} selon lequel,
si $W$ est une repr\'esentation admissible de $\Gamma_K$, alors $W^{\la}$ est dense dans $W$.
On ne peut pas utiliser ce r\'esultat ici car
une $\hat{K}_\infty$-repr\'esentation semi-lin\'eaire de $\Gamma_K$ n'est pas une repr\'esentation
admissible de $\Gamma_K$.  Par exemple, dans le cas de l'extension cyclotomique,
$(\OO_{\hat{K}_\infty}/p)^{\Gamma_K}$ contient l'image de $\frac{p}{\zeta-1}$ modulo $p$, pour toute racine de l'unit\'e~$\zeta$ d'ordre une puissance de $p$, et donc est de dimension infinie sur~$\Fp$.  
\end{rema}

Soit $d$ la dimension de $\Gamma_K$; il existe alors (\S 27 de \cite{PSLG})
un groupe analytique ${\mathbb G}$ de dimension~$d$, d\'efini sur $\Qp$, tel que l'on ait
$\Gamma_K={\mathbb G}(\Zp)$.  
Si $n\geq 1$, on note $\Gamma_n$ le groupe ${\mathbb G}(p^n\Zp)$, image de $p^n{\mathfrak g}$ 
par l'exponentielle
(o\`u ${\mathfrak g}$ est l'alg\`ebre de Lie de $\Gamma_K$ et est un $\Zp$-module libre de rang~$d$),
et on note $K_n$ le sous-corps $K_\infty^{\Gamma_n}$ de
$K_\infty$.  L'anneau $\hat{K}_\infty^\la$ est la limite inductive des 
$\hat{K}_\infty^{\Gamma_n\dan}$, o\`u l'on a not\'e
$\hat{K}_\infty^{\Gamma_n\dan}$ l'ensemble des $v$ tels que $g\mapsto g\cdot v$ soit
analytique sur~$\Gamma_n$, et $\hat{K}_\infty^{\Gamma_n\dan}$ est une $K_n$-alg\`ebre de Banach;
on note $X_n$ le $K_n$-espace analytique qu'elle d\'efinit.  
Le r\'esultat suivant (th.~\ref{sen1}) 
montre que, apr\`es extension des scalaires \`a $\Cp$, 
$X_n$ devient une boule de dimension~$d-1$.

\begin{theo}\label{thgen}
Soit $V$ une repr\'esentation fid\`ele de $\Gamma_K$, et soit
${\mathbb H}$ le sous-groupe \`a un param\`etre de ${\mathbb G}$ engendr\'e par l'op\'erateur
de Sen de $V$,
vu comme \'el\'ement de $\Cp\otimes_{\Zp} {\mathfrak g}$.  Si $n \geq 1$, alors
$X_n(\Cp)={\mathbb H}(p^n\OO_{\Cp})\backslash {\mathbb G}(p^n\OO_{\Cp})$.
\end{theo}

\begin{rema}\label{thgen2}
{\rm (i)} La preuve du th\'eor\`eme montre que ${\mathbb H}$ n'est pas trivial,
ce qui se traduit par le fait que $\Theta_{\rm Sen}\neq 0$.
Comme la seule hypoth\`ese que l'on a faite sur $\Gamma_K$ est que
son sous-groupe d'inertie est infini, cela fournit une preuve du th\'eor\`eme
de Sen (troisi\`eme point du (ii) de la rem.~\ref{rem2})
selon lequel une repr\'esentation de $G_K$ 
est de Hodge-Tate \`a poids de Hodge-Tate tous
nuls si et seulement si le sous-groupe d'inertie de $G_K$ agit \`a travers un quotient fini;
cette preuve n'utilise pas les r\'esultats de~\cite{S72}. 
Par contre, il n'a pas l'air possible de retrouver, par cette m\'ethode,
 le r\'esultat plus fin
d\'ecrivant l'alg\`ebre de Lie de $G_K$ en termes de $\Theta_{\rm Sen}$.

{\rm (ii)} La m\'ethode permettant de prouver le th.~\ref{thS1} peut aussi \^etre
utilis\'ee pour montrer que si $L$ est une extension galoisienne de $K$ qui n'est pas profond\'ement ramifi\'ee, alors $H^1(\Gal(L/K),\GL_d(\hat L))=\{1\}$.
La non nullit\'e de $\Theta_{\rm Sen}$ implique donc que $K_\infty/K$ est profond\'ement
ramifi\'ee.  Autrement dit, on a red\'emontr\'e, sans utiliser~\cite{S72},
qu'une extension galoisienne de groupe de Galois un groupe de Lie $p$-adique
est profond\'ement ramifi\'ee si et seulement si elle est infiniment ramifi\'ee.

{\rm (iii)} Comme $X_n$ est d\'efini sur $K_n$ et que ${\mathbb G}$ est d\'efini sur $\Qp$,
on dispose d'actions de $G_{K_n}$ sur $X_n(\Cp)$ et sur ${\mathbb G}(p^n\OO_{\Cp})$, mais
celles-ci ne sont pas compatibles: il faut tordre l'application naturelle.
Notons $\gamma:G_{K_n} \to\Gamma_n = {\mathbb G}(p^n\Zp)$ la projection
naturelle et $\pi:{\mathbb G}(p^n\OO_{\Cp})\to X_n(\Cp)$ l'application fournie
par le th.~\ref{thgen}. On a alors
$$\sigma(\pi(x))=\pi(\gamma(\sigma)\sigma(x)),\quad{\text{si $x\in {\mathbb G}(p^n\OO_{\Cp})$
et $\sigma\in G_{K_n}$.}}$$
Le fait que cette action descende \`a 
${\mathbb H}(p^n\OO_{\Cp})\backslash {\mathbb G}(p^n\OO_{\Cp})$
est une cons\'equence de ce que $\Theta_{\rm Sen}$ commute \`a $G_{K_n}$.
On remarquera que l'orbite sous $G_{K_n}$
de l'\'element neutre de ${\mathbb G}$ est Zariski dense dans
$X_n(\Cp)$ (c'est l'image par $\pi$ de $\Gamma_n$), ce qui explique
que $\hat{K}_\infty^{\la}=\cup_{n\geq 1} \hat{K}_\infty^{\Gamma_n\dan}$ soit un corps
bien que $\cap_{n\geq1}X_n(\Cp)$ ne soit pas vide.
\end{rema}

\subsection{Extensions de type Lubin-Tate}
Supposons dans ce qui suit que
$K_\infty$ est l'extension de $K$ engendr\'ee par les points de torsion 
d'un $\OO_F$-module formel, avec $F \subset K$ extension galoisienne de $\Qp$. 
Par la th\'eorie de Lubin-Tate, le groupe $\Gamma_K$ s'identifie, via le caract\`ere 
de Lubin-Tate $\chi_F$, \`a un sous-groupe ouvert de $\OO_F^\times$. La th\'eorie 
des p\'eriodes $p$-adiques fournit, pour chaque plongement $\tau : F \to \Qpbar$ 
diff\'erent de l'identit\'e, un \'el\'ement $u_\tau \in \Cp^\times$ tel 
que $g(u_\tau) = \tau \circ \chi_F(g) \cdot u_\tau$. Soit $x_\tau = \log(u_\tau)$. 
Le r\'esultat suivant donne une bonne id\'ee de ce \`a quoi ressemble $\hat{K}_\infty^{\la}$;
nous renvoyons au 
th.~\ref{qpanfinf} pour un \'enonc\'e plus pr\'ecis mais plus technique.

\begin{theo}\label{thmC}
L'anneau $K_\infty[\{x_\tau\}_{\tau \neq \Id}]$ est dense dans $\hat{K}_\infty^{\la}$ pour sa topologie naturelle.
\end{theo}

Si $W$ est une $\hat{K}_\infty$-repr\'esentation semi-lin\'eaire de dimension finie 
de $\Gamma_K$, alors pour tout plongement $\tau : F \to \Qpbar$, on dispose de l'op\'erateur 
diff\'erentiel $\nabla_\tau : W^{\la} \to W^{\la}$. Un vecteur $w \in W$ est 
localement $F$-analytique s'il est localement $\Qp$-analytique et si $\nabla_\tau(w)=0$ 
pour tout $\tau \neq \Id$. On note $W^{\fan}$ l'espace des vecteurs de $W$ qui sont 
localement $F$-analytiques. Le th.~\ref{thmC} implique que $\hat{K}_\infty^{\fan} = K_\infty$. 
On a alors (th.~\ref{fanvec}).

\begin{theo}\label{thmD}
Si $K_\infty/K$ est une extension de Lubin-Tate, et si $W$ est une $\hat{K}_\infty$-repr\'esentation semi-lin\'eaire de dimension finie de $\Gamma_K$, alors $\hat{K}_\infty^{\la} \otimes_{K_\infty} W^{\fan} \to W^{\la}$ est un isomorphisme.
\end{theo}

On peut donc associer \`a $W$ le $K_\infty$-espace vectoriel $W^{\fan}$, qui est alors muni de l'op\'erateur diff\'erentiel $K_\infty$-lin\'eaire $\nabla_{\Id}$. Cet op\'erateur est semblable \`a $\Theta_{\mathrm{Sen}}$.

\section{Rappels et compl\'ements}
\label{rappelsec}

Dans tout cet article, nous utilisons la notation multi-indice. Soit $\NN = \ZZ_{\geq 0}$; si $\cbf = (c_1, \hdots,c_d)$ et $\kbf = (k_1,\hdots,k_d) \in \NN^d$, alors $\cbf^\kbf = c_1^{k_1} \times \cdots \times c_d^{k_d}$. On pose $\kbf ! = k_1! \times \cdots \times k_d!$ et $|\kbf|= k_1 + \cdots + k_d$. On note $\mathbf{1}_j$ le $d$-uplet $(k_1,\hdots,k_d)$ o\`u $k_i = 0$ si $i \neq j$ et $k_j=1$. Rappelons qu'un espace LB est un espace vectoriel topologique qui est la limite inductive d'une suite d'espaces de Banach, et que le th\'eor\`eme de l'image ouverte est vrai pour ces espaces (th.~1.1.17 de \cite{MEAMS}).

\subsection{Vecteurs localement analytiques}\label{locvec}
Nous faisons tout d'abord quelques rappels et compl\'ements sur les vecteurs localement analytiques. Nous renvoyons par exemple \`a la monographie \cite{MEAMS} pour plus de d\'etails. 

Soit $G$ un groupe de Lie $p$-adique, et soit $W$ un $\Qp$-espace de Banach qui est une repr\'esentation de $G$. Si $w \in W$, alors suivant le  \S\,7 de \cite{ST03}, disons que $w$ est 
{\it localement $\Qp$-analytique} si l'application orbite $G \to W$, donn\'ee par $g \mapsto g ( w)$, 
est une fonction localement $\Qp$-analytique sur $G$. On note $W^{\la}$ l'espace des vecteurs 
localement $\Qp$-analytiques de $W$. Si $w \in W^{\la}$, il existe donc un sous-groupe compact ouvert $H_w$ 
de $G$, et des coordonn\'ees locales $c_i : H_w \to \Zp$ qui donnent lieu \`a une bijection analytique 
entre $H_w$ et $\Zp^d$, et une suite $\{ w_\kbf \}_{\kbf \in \NN^d}$ de $W$ telle que $w_{\kbf} \to 0$ 
quand $|\kbf| \to +\infty$, tels que $h(w) = \sum_{\kbf \in \NN^d} \cbf(h)^{\kbf} w_{\kbf}$ si $h \in H_w$. 
Si $H$ est un sous-groupe de $G$, on note $W^{H\dan}$ l'espace des vecteurs de $W$ qui sont globalement analytiques sur $H$.

Soient $G$ et $H$ deux groupes de Lie $p$-adiques, et $f : G \to H$ un morphisme 
analytique de groupes. 
Si $W$ est une repr\'esentation de $H$, on peut aussi voir $W$ 
comme une repr\'esentation de $G$.  Le r\'esultat suivant est imm\'ediat.

\begin{lemm}
\label{lamap}
Si $w \in W$ est un vecteur localement $\Qp$-analytique pour $H$, alors c'est un vecteur localement $\Qp$-analytique pour $G$.
\end{lemm}

\begin{lemm}
\label{mapla}
Soit $G$ un groupe de Lie $p$-adique, soient $W$ et $X$ deux $\Qp$-espaces de Banach, et soit $\pi : W \to X$ une application lin\'eaire continue. Si $f : G \to W$ est une fonction  localement $\Qp$-analytique, alors $\pi \circ f : G \to X$ est localement $\Qp$-analytique.
\end{lemm}

\begin{proof}
Si $f(g) = \sum_{\kbf \in \NN^d} \cbf(g)^{\kbf} f_{\kbf}$, alors $\pi \circ f (g) = \sum_{\kbf \in \NN^d} \cbf(g)^{\kbf} \pi(f_{\kbf})$.
\end{proof}

\begin{prop}
\label{basan}
Soit $G$ un groupe de Lie $p$-adique, et soient $W$ et $B$ deux repr\'esentations de $G$. Si $B$ est un anneau et si $W$ est un $B$-module libre de rang fini, admettant une base $w_1,\hdots,w_d$ telle que $g \mapsto \Mat(g)$ est une fonction globalement analytique $G \to \GL_d(B) \subset \mathrm{M}_d(B)$, alors 

{\rm (i)} $W^{H\dan} = \oplus_{i=1}^d B^{H\dan} \cdot w_i$ si $H$ est un sous-groupe de $G$;

{\rm (ii)} $W^{\la} = \oplus_{i=1}^d B^{\la} \cdot w_i$.
\end{prop}

\begin{proof}
Le (ii) \'etant une cons\'equence imm\'ediate du (i), il suffit de prouver le (i).
L'inclusion de $\oplus_{i=1}^d B^{H\dan} \cdot w_i\subset W^{H\dan}$ est \'evidente;
montrons l'autre inclusion. Si $w \in W$, on peut \'ecrire $w= \sum_{i=1}^d b_i w_i$. Soit  $f_i : W \to B$ la fonction $w \mapsto b_i$. \'Ecrivons $\Mat(g)=(m_{i,j}(g))_{i,j}$. Si $h \in H$, alors $h(w) = \sum_{i,j=1}^d h(b_i) m_{i,j}(h) w_j$ et si $w \in W^{H\dan}$, alors $h \mapsto f_j(h(w)) = \sum_{i=1}^d h(b_i) m_{i,j}(h)$ est une fonction globalement analytique $H \to B$ par le lem.~\ref{mapla}. Si $\Mat(h)^{-1}=(n_{i,j}(h))_{i,j}$, alors $h(b_i) = \sum_{j=1}^d f_j(h(w)) n_{i,j}(h)$ et on en d\'eduit bien que  $b_i \in B^{H\dan}$.
\end{proof}

On se donne \`a pr\'esent un sous-groupe compact ouvert $G_1$ de $G$, tel que si l'on pose $G_n = G_1^{p^{n-1}}$ pour $n \geq 1$, alors $G_n$ est un sous-groupe de $G_1$ et tel qu'il existe des coordonn\'ees locales $c_i : G_1 \to \Zp$ telles que $\cbf(G_n) = (p^n \Zp)^d$ pour tout $n \geq 1$. Un tel sous-groupe existe toujours : il suffit par exemple de se donner un sous-groupe compact ouvert $G_0$ de $G$ qui est $p$-valu\'e et satur\'e (cf. le \S\,23 de \cite{PSLG} pour la d\'efinition, et le \S\,27 pour l'existence d'un tel $G_0$) et de poser $G_n=G_0^{p^n}$ (cf.\ les \SS\,23 et 26 de ibid). 

Si $w \in W^{\la}$, il existe $n \geq 1$ tel que $w \in W^{G_n\dan}$ et on peut \'ecrire $g(w) = \sum_{\kbf \in \NN^d} \cbf(g)^{\kbf} w_{\kbf}$ si $g \in G_n$ o\`u $\{ w_{\kbf} \}_{\kbf \in \NN^d}$ est une suite de $W$ telle que $p^{n|\kbf|} w_{\kbf} \to 0$. On pose alors $\|w\|_{G_n} = \sup_{\kbf \in \NN^d} \| p^{n|\kbf|} \cdot w_{\kbf}\|$. Cette norme co\"\i ncide 
avec la norme d\'eduite de l'inclusion $W^{G_n\dan} \to {\mathcal C}^{\an}(G_n,W)$ et ne d\'epend donc pas du choix de coordonn\'ees locales. Par le corollaire 3.3.6 de \cite{MEAMS}, l'inclusion $(W^{G_n\dan})^{G_n\dan} \to W^{G_n\dan}$ est un isomorphisme topologique. Ceci implique en particulier que si $w \in W^{G_n\dan}$, alors $w_\kbf \in W^{G_n\dan}$ pour tout $\kbf \in \NN^d$. 

Le lemme ci-dessous est une cons\'equence imm\'ediate des d\'efinitions.

\begin{lemm}\label{descnm}
Si $w \in W^{G_n\dan}$, alors 

{\rm (i)} $w \in W^{G_m\dan}$ pour tout $m \geq n$, 

{\rm (ii)} $\|w\|_{G_{m+1}} \leq \|w\|_{G_m}$ si $m \geq n$,

{\rm (iii)} $\|w\|_{G_m} = \| w \|$ si $m \gg 0$.
\end{lemm}

L'espace $W^{G_n\dan}$, muni de $\|\cdot\|_{G_n}$, est un espace de Banach. 
On a $W^{\la} = \cup_{n \geq 1} W^{G_n\dan}$ et $W^{\la}$ est muni de la topologie de la limite inductive ce qui en fait un espace LB.

\begin{lemm}\label{prodwan}
Si $W$ est un anneau, tel que $\|xy\| \leq \|x\| \cdot \|y\|$ si $x,y \in W$, alors 

{\rm (i)}
 $W^{G_n\dan}$ est un anneau et $\|xy\|_{G_n} \leq \|x\|_{G_n} \cdot \|y\|_{G_n}$ si $x,y \in W^{G_n\dan}$,

{\rm (ii)}
 si de plus $W$ est un corps, alors $W^{\la}$ est aussi un corps.
\end{lemm}

\begin{proof}
\'Ecrivons $g(x) = \sum_{\kbf \in \NN^d} \cbf(g)^\kbf x_\kbf$ et $g(y) = \sum_{\kbf \in \NN^d} \cbf(g)^\kbf y_\kbf$. On a alors \[ g(xy) = \sum_{\kbf \in \NN^d} \cbf(g)^\kbf \left( \sum_{\ibf+\jbf=\kbf} x_\ibf y_\jbf \right), \]
et si $\ibf+\jbf=\kbf$, alors $\|p^{n|\kbf|} x_\ibf \cdot y_\jbf \| \leq \|p^{n|\ibf|} x_\ibf \| \cdot \|p^{n|\jbf|} y_\jbf\|$, ce qui implique le (i). 

Pour le (ii), soit $w \in W^{G_n\dan} \setminus \{ 0 \}$ avec $g(w) = \sum_{\kbf \in \NN^d} \cbf(g)^\kbf w_\kbf$. On a 
\[ \frac{1}{g(w)} = \frac{1}{w + \sum_{\kbf \neq 0^d} \cbf(g)^\kbf w_\kbf}  = \frac{1}{w} \cdot \frac{1}{1 + \sum_{\kbf \neq 0^d} \cbf(g)^\kbf \cdot w_\kbf/w}. \]
Ceci implique que $1/w \in W^{G_m \dan}$ d\`es que $m \geq n$ est suffisament grand pour que l'on ait $\sup_{\kbf \neq 0^d} |p^{m|\kbf|} \cdot w_\kbf/w| < 1$ de sorte que $g(1/w) = \sum_{j \geq 0} (-1)^j  (\sum_{\kbf \neq 0^d} \cbf(g)^\kbf w_\kbf/w)^j /w$.
\end{proof}

L'alg\`ebre de Lie de $G$, $\Lie(G)$ agit sur $W^{\la}$.

\begin{lemm}\label{liecont}
Si $D \in \Lie(G)$ et $n \geq 1$, alors $D(W^{G_n \dan}) \subset W^{G_n \dan}$, et il existe une constante $C_D$ telle que $\| D(x) \|_{G_n} \leq C_D \| x \|_{G_n}$ si $x \in W^{G_n \dan}$.
\end{lemm}

\begin{proof}
L'espace $W^{G_n\dan}$ est un espace de Banach muni d'une action localement $\Qp$-analytique de $G$ et le lemme r\'esulte alors de la prop.~3.2 de \cite{STAM}.
\end{proof}

\subsection{L'anneau $\bsen$ et la th\'eorie de Sen}\label{bsensec}

Nous donnons ici une l\'eg\`ere g\'en\'eralisation de la construction de l'anneau $\bsen$ de \cite{PCsen}. Soit $\chi : G_K \to \Zp^\times$ un caract\`ere (dans \cite{PCsen}, $\chi$ est le  caract\`ere cyclotomique). Soit $H_K = \ker \chi$ et $K_\infty = \Qpbar^{H_K}$ et $\Gamma_n = \{ g \in \Gal(K_\infty/K)$ tels que $\chi(g) \in 1+p^n \Zp\}$ et soit $K_n = K_\infty^{\Gamma_n}$.

Soit $u$ une variable et soit $\bsen$ l'ensemble des s\'eries formelles $f(u) = \sum_{i \geq 0} a_i u^i$ telles que $a_i \in \Cp$ et qui ont un rayon de convergence non nul. Si $n \geq 1$, l'anneau $\bsen^n$ est l'ensemble des s\'eries formelles de rayon de convergence $p^{-n}$ de telle sorte que $\bsen = \cup_{n \geq 1} \bsen^n$. 
On munit $\bsen^n$ d'une action de $G_{K_n}$ gr\^ace \`a la formule
 $$g(\sum_{i \geq 0} a_i u^i) = \sum_{i \geq 0} g(a_i) (u+\log \chi(g))^i.$$

\begin{theo}\label{pcsenth}
On a $(\bsen^n)^{G_{K_n}} = K_n$.
\end{theo}

\begin{proof}
La d\'emonstration est tout \`a fait analogue \`a celle du (i) du th.~2 de \cite{PCsen} (l'affirmation correspondante quand $\chi$ est le caract\`ere cyclotomique). Nous en rappelons bri\`evement les \'etapes principales. Par le th\'eor\`eme d'Ax-Sen-Tate, on a $\Cp^{H_K} = \hat{K}_\infty$ de telle sorte que si $f(u) = \sum_{i \geq 0} a_i u^i \in (\bsen^n)^{G_{K_n}}$, alors $a_i \in \hat{K}_\infty$. Le fait que $g(f)=f$ implique que pour tout $i \geq 0$, on a $g(a_i) = \sum_{j \geq 0} a_{i+j} \binom{i+j}{i}(-\log \chi(g))^j$. Soient $R_m : \hat{K}_\infty \to K_m$ les traces de Tate normalis\'ees pour $m \geq n$ (cf.\  le \S\,3.1 de \cite{T67}). En appliquant $R_m$ \`a l'\'equation pr\'ec\'edente, on obtient
\[ g(R_m(a_i)) = \sum_{j \geq 0} R_m(a_{i+j}) \binom{i+j}{i}(-\log \chi(g))^j. \]
Le membre de gauche est une fonction de $g$ qui est constante sur $\Gamma_m$, alors que le membre de droite est une fonction analytique sur $\Gamma_n$. Ces fonctions sont donc constantes sur $\Gamma_n$, ce qui fait que  $R_m(a_k)=0$ pour tout $m \geq n$ et $k \geq 1$. Ceci implique que $f=a_0 \in K_n$.
\end{proof}

Si $K_\infty/K$ est l'extension cyclotomique, on peut utiliser l'anneau $\bsen^n$
pour retrouver~\cite{PCsen} les invariants de Sen des $\Qp$-repr\'esentations de $G_K$ (rem.~\ref{rem2}).
Soit $V$ une $\Qp$-repr\'esentation de dimension finie de $G_K$.
Si $n\in\NN$, on note ${\rm D}_{{\rm Sen},n}'(V)$ le $K_n$-espace vectoriel
$${\rm D}_{{\rm Sen},n}'(V)=(\bsen^n\otimes_{\Qp} V)^{G_{K_n}},$$
et on pose $\dsen'(V)=\cup_{n\in\NN}{\rm D}_{{\rm Sen},n}'(V)$.
On peut \'ecrire un \'el\'ement de $\bsen^n$ sous la forme $\alpha^{(0)}+
\alpha^{(1)}u+\cdots$, o\`u les $\alpha^{(i)}$ sont des \'el\'ements de $\Cp$.
Cela permet d'\'ecrire un \'el\'ement $\delta$ de $\dsen'(V)$
sous la forme $\delta^{(0)}+
\delta^{(1)}u+\cdots$, o\`u les $\delta^{(i)}$ sont des \'el\'ements 
de $\Cp\otimes_{\Qp} V$.  On a alors le r\'esultat suivant.

\begin{prop}\label{classique3}
{\rm (i)}  L'application
$\delta\mapsto\delta^{(0)}$ induit un isomorphisme de $K(\mu_{p^\infty})$-espaces
vectoriels de $\dsen'(V)$ sur $\dsen(V)$ et de $K_n$-espaces vectoriels
de ${\rm D}_{{\rm Sen},n}'(V)$ sur ${\rm D}_{{\rm Sen},n}(V)$ si $n$ est assez grand.  

{\rm (ii)} L'isomorphisme inverse
est $d\mapsto\iota(d)=e^{-u\Theta_{\rm Sen}}\cdot d$.
\end{prop}
\section{Th\'eorie de Sen et vecteurs localement analytiques}
\label{gensensec}

Dans ce chapitre, nous montrons comment g\'en\'eraliser la th\'eorie de Sen en consid\'erant les vecteurs localement analytiques au lieu des vecteurs finis.

\subsection{Vecteurs finis et vecteurs localement analytiques}
\label{vfvlasec}

Soit $W$ une $\hat{K}_\infty$-repr\'esentation semi-lin\'eaire de dimension finie de $\Gamma_K$. Si $w \in W$, rappelons que l'on dit que $w$ est {\it $K$-fini}
 s'il appartient \`a un sous $K$-espace vectoriel de dimension finie de $W$ qui est stable 
sous $\Gamma_K$. Soit $W^{\fin}$ l'ensemble des vecteurs $K$-finis de $W$. 
On a le r\'esultat suivant d\^u \`a Sen \cite{SN80}.

\begin{prop}
\label{finok}
Si $\dim\Gamma_K=1$, et si $W$ est une $\hat{K}_\infty$-repr\'e\-sentation semi-lin\'eaire de dimension finie de $\Gamma_K$, l'application $\hat{K}_\infty \otimes_{K_\infty} W^{\fin} \to W$ est un isomorphisme.
\end{prop}

Nous donnons \`a pr\'esent le lien entre $W^{\fin}$ et les vecteurs localement analytiques de $W$.

\begin{theo}
\label{finisan}
Si $\dim\Gamma_K=1$, et si $W$ est une $\hat{K}_\infty$-repr\'e\-sentation semi-lin\'eaire de dimension finie de $\Gamma_K$, alors $W^{\fin} = W^{\la}$.
\end{theo}

\begin{proof}
Montrons d'abord l'\'egalit\'e pour $W=\hat{K}_\infty$. Soient $R_n : \hat{K}_\infty \to K_n$ les traces de Tate normalis\'ees (cf.\ le \S\,3.1 de \cite{T67}), de telle sorte que si $x \in \hat{K}_\infty$, alors $x = \lim_{n \to +\infty} R_n(x)$. Les applications $\{R_n\}_{n \gg 0}$ commutent avec $\Gamma_K$. Soit $x$ un vecteur $\Qp$-analytique pour un sous-groupe $\Gamma_m \subset \Gamma_K$. Si $k \geq 1$, alors $R_{m+k}(x)$ est un vecteur $\Qp$-analytique pour $\Gamma_m$ et un vecteur constant pour $\Gamma_{m+k}$, et c'est donc un vecteur constant pour $\Gamma_m$. Ceci implique que $R_{m+k}(x) \in K_m$ pour tout $k \geq 1$, et donc que $x \in K_m$.

Revenons \`a pr\'esent au cas g\'en\'eral. Si $w \in W^{\fin}$, alors $w$ vit dans une repr\'esentation 
$\Qp$-lin\'eaire de dimension finie $V_w$ de $\Gamma_K$. Par le th\'eor\`eme de 
Cartan (cf.\ la prop.~3.6.10 de \cite{MEAMS} ou le \S\,V.9 de \cite{SLALG}), les vecteurs de $V_w$ sont 
localement $\Qp$-analytiques,
et le lem.~\ref{mapla} implique que $w \in W^{\la}$. On a donc $W^{\fin}\subset W^{\la}$.

Pour montrer l'inclusion dans l'autre sens, fixons une base $e_1,\hdots,e_d$ de $W^{\fin}$ sur $K_\infty$.
 Les $e_i$ forment une base de $W$ sur $\hat{K}_\infty$ par le th\'eor\`eme de Sen, et la prop.~\ref{basan} 
implique que $W^{\la} = \oplus_{i=1}^d \hat{K}^{\la}_\infty \cdot e_i$. Comme $\hat{K}^{\la}_\infty = K_\infty$, on a $W^{\la} = W^{\fin}$.
\end{proof}

\begin{rema}\label{dsenn}
{\rm (i)}
 Si $\Gamma_K$ et $W$ sont comme ci-dessus, les prop.~\ref{basan} et \ref{finok} et le th.~\ref{finisan} impliquent que pour $n$ assez grand, 
$W^{\Gamma_n\dan}$ est un $K_n$-espace vectoriel stable par $\Gamma_n$
et tel que $W = \hat{K}_\infty \otimes_{K_n} W^{\Gamma_n\dan}$. 
On en d\'eduit que, si $W=(\Cp\otimes_{\Qp}V)^{H_K}$
o\`u $V$ est une $\Qp$-repr\'esentation de $G_K$, cet espace n'est autre que
l'espace ${\rm D}_{{\rm Sen},n}(V)$ du (iii) de la rem.~\ref{rem2}.

{\rm (ii)} On peut retrouver le module ${\rm D}_{{\rm dif},n}^+(V)$ de la m\^eme mani\`ere,
comme limite projective des $W_k^{\Gamma_n\dan}$, avec $W_k=((\bdr^+/t^k)\otimes_{\Qp} V)^{H_K}$. 
\end{rema}

\subsection{Une g\'en\'eralisation de la th\'eorie de Sen}
\label{senrem}

Nous donnons maintenant notre g\'en\'eralisation du th\'eor\`eme de Sen, quand $\Gamma_K$ est un groupe de Lie $p$-adique de dimension quelconque, en utilisant $W^{\la}$ et non plus  $W^{\fin}$. Rappelons que $\hat{K}_\infty^{\la}$ est un corps par le (2) du lem.~\ref{prodwan}.

\begin{theo}
\label{anisok}
Si $W$ est une $\hat{K}_\infty$-repr\'esentation semi-lin\'eaire de dimension finie de $\Gamma_K$, alors l'application $\hat{K}_\infty \otimes_{\hat{K}_\infty^{\la}} W^{\la} \to W$ est un isomorphisme.
\end{theo}

Commen\c{c}ons par montrer que l'on peut remplacer $K$ par une extension finie.

\begin{lemm}
\label{ladescet}
Si $W$ est une $\hat{K}_\infty$-repr\'esentation semi-lin\'eaire de dimension finie de $\Gamma_K$, si $L/K$ est une extension galoisienne finie, si $W_L = \hat{L}_\infty \otimes_{\hat{K}_\infty} W$, et si $W_L^{\la}$ est un $\hat{L}_\infty^{\la}$-espace vectoriel de dimension finie, alors $W_L^{\la} = \hat{L}_\infty^{\la} \otimes_{\hat{K}_\infty^{\la}} W^{\la}$.
\end{lemm}

\begin{proof}
L'extension $\hat{L}_\infty^{\la} / \hat{K}_\infty^{\la}$ est une extension galoisienne finie par la prop.~\ref{basan}. On a  $W^{\la} = (W_L^{\la})^{\Gal(\hat{L}_\infty^{\la} / \hat{K}_\infty^{\la})}$ et le lemme r\'esulte de la descente \'etale (cf.\  par exemple le \S\,2.2 de \cite{LB11}).
\end{proof}

\begin{proof}[D\'emonstration du th.~\ref{anisok}]
Soient $L_\infty = K_\infty(\mu_{p^\infty})$ et $\Delta_K = \Gal(L_\infty/K_\infty)$. Quitte \`a remplacer $K$ par une extension finie, on peut supposer ou bien que $L_\infty=K_\infty$ (et donc $\Delta_K=\{1\}$) ou bien que le caract\`ere cyclotomique $\chi_{\cyc}$ induit un isomorphisme de $\Delta_K$ sur $1+2p^n\Zp$, avec $n\geq 1$.

Soit $X= \hat{L}_\infty \otimes_{\hat{K}_\infty} W$ et $\dsen(W) = (X^{\Gal(L_\infty/K(\mu_{p^\infty}))})^{\fin}$ 
le module de Sen cyclotomique associ\'e \`a $W$. Par la th\'eorie de Sen classique, 
on a $X = \hat{L}_\infty \otimes_{K(\mu_{p^\infty})} \dsen(W)$ ce qui fait que 
$X^{\la} = \hat{L}_\infty^{\la} \otimes_{K(\mu_{p^\infty})} \dsen(W)$ par la prop.~\ref{basan}. On a donc 
\[ W^{\la} = (\hat{L}_\infty^{\la} \otimes_{K(\mu_{p^\infty})} \dsen(W))^{\Delta_K}. \]
Si $L_\infty=K_\infty$ (et donc $\Delta_K=\{1\}$), cela permet de conclure.
Si $L_\infty\neq K_\infty$, nous aurons besoin du r\'esultat suivant.

\begin{lemm}\label{zla}
Quitte \`a remplacer $K$ par une extension finie, on peut trouver
 $z\in \hat{L}_\infty^{\la}$ tel que
$g(z) = z + \log \chi_{\cyc}(g)$ pour tout $g\in \Delta_K$.
\end{lemm}

Admettons le lemme et terminons la d\'emonstration du th\'eor\`eme.
Choisissons $z_0\in L_\infty$ assez proche de $z$ pour que la s\'erie $\exp((z_0-z)\Theta_{\rm Sen})$
converge vers un op\'erateur de $\dsen(W)$.  Choisissons aussi une base $e_1,\dots,e_d$ de $\dsen(W)$
sur $K(\mu_{p^\infty})$.  Quitte \`a augmenter $K$, on peut supposer
que $z_0\in K$, et que $f_i=\exp((z_0-z)\Theta_{\rm Sen})\cdot e_i$ est fixe par $\Delta_K$
pour tout $i$ (en effet, $f_i$ est tu\'e par $\Theta_{\rm Sen}$ et donc fixe par un sous-groupe ouvert
de $\Delta_K$ que l'on peut supposer \^etre \'egal \`a $\Delta_K$ quitte \`a augmenter $K$).
On a alors $\hat{L}_\infty^{\la} \otimes_{K(\mu_{p^\infty})} \dsen(W)=\oplus_{i=1}^d\hat{L}_\infty^{\la} \cdot f_i$ avec $f_i \in W^{\la}$,
et donc
\[X^{\la}= \hat{L}_\infty^{\la} \otimes_{K(\mu_{p^\infty})} \dsen(W) 
= \hat{L}_\infty^{\la} \otimes_{\hat{K}_\infty^{\la}} (\hat{L}_\infty^{\la} \otimes_{K(\mu_{p^\infty})} \dsen(W))^{\Delta_K}
=\hat{L}_\infty^{\la} \otimes_{\hat{K}_\infty^{\la}} W^{\la}. \]
Ceci implique que $X = \hat{L}_\infty \otimes_{\hat{K}_\infty^{\la}} W^{\la}$ et,
en prenant les points fixes sous $\Delta_K$, que $W = \hat{K}_\infty \otimes_{\hat{K}_\infty^{\la}} W^{\la}$.
\end{proof}

\begin{proof}[D\'emonstration du lem.~\ref{zla}]
Soit $V$ une $\Qp$-repr\'esentation fid\`ele de $\Gamma_K$, de dimension finie
(si $\dim V=d$, cela \'equivaut, modulo le choix d'une base de $V$ \`a se donner
une injection de $\Gamma_K$ dans $\GL_d(\Qp)$).
On peut voir $V$ comme une repr\'esentation de $G_K$ et il r\'esulte de l'isomorphisme
ci-dessus, appliqu\'e \`a $W=\hat{K}_\infty\otimes_{\Qp} V$, que 
$$\hat{K}_\infty^\la\otimes_{\Qp} V=(\hat{L}_\infty^\la\otimes_{K(\mu_{p^\infty})}\dsen(V))^{\Delta_K}.$$
L'op\'erateur $\Theta_{\rm Sen}$ sur $\dsen(V)$ n'est pas nul
et, quitte \`a remplacer $K$ par une extension finie, 
on peut supposer que $K$ contient ses valeurs
propres.  Il y a alors deux cas:

$\bullet$ $\Theta_{\rm Sen}$ a une valeur propre non nulle $s$.  En d\'ecomposant une base de $V$
sur une base de $\dsen(V)$ dans laquelle la matrice de $\Theta_{\rm Sen}$ est sous forme de Jordan,
on en d\'eduit l'existence de $x_s\in \hat{L}_\infty^\la$ tel que l'on ait
$\chi_{\rm cycl}(g)^s g(x_s)=x_s$, pour tout $g$ dans un sous-groupe ouvert
de $\Delta_K$ (que l'on peut supposer \^etre \'egal \`a $\Delta_K$ en rempla\c cant $K$
par une extension finie).  Si on \'ecrit $x_s$ sous la forme $x_s^0(1+y)$, avec $|y|_p<1$
et $\log x_s^0=0$, alors la s\'erie $-\frac{1}{s}\log(1+y)$ converge dans $\hat{L}_\infty^\la$
d'apr\`es le lemme~\ref{descnm}, et sa somme v\'erifie les propri\'et\'es voulues.

$\bullet$ Toutes les valeurs propres de $\Theta_{\rm Sen}$ sont nulles.
 La d\'ecomposition d'une base
de $V$ comme ci-dessus fournit alors directement un \'el\'ement tel que
$g(z) = z + \log \chi_{\cyc}(g)$ pour tout $g$ dans un sous-groupe ouvert de $\Delta_K$.
\end{proof}

\subsection{Remarques sur la th\'eorie de Schneider-Teitelbaum}
D'apr\`es~\cite{ST03}, si on se place dans la cat\'egorie des repr\'esentations admissibles
de $\Gamma_K$, le foncteur $\Pi\mapsto\Pi^{\la}$ a de bonnes propri\'et\'es : il est
exact et $\Pi^{\la}$ est dense dans $\Pi$.  Si on sort du cadre admissible, la situation
est nettement moins agr\'eable, comme on peut le prouver en utilisant le th.~\ref{finisan}.

Pla\c{c}ons-nous dans le cas de l'extension cyclotomique et posons $U=(\bcris^+)^{\varphi=p}$. La suite exacte $0\to\Qp(1)\to U\to\Cp\to 0$ induit,
en prenant les points fixes sous l'action de $H_K$, la suite exacte
$$0\to \Qp(1)\to U^{H_K}\to \hat K_\infty\to H^1(H_K,\Qp(1))\to 0.$$
Or $(U^{H_K})^{\la}$ est r\'eduit \`a $\Qp(1)$ (en effet, si $x\in U^{H_K}$ est localement
analytique, son image dans $\hat{K}_\infty$ appartient \`a $K_\infty$ et donc est
fixe par un \'el\'ement $\gamma\neq 1$ de $\Gamma_K$, et comme la valeur propre de $\gamma$ sur
$\Qp(1)$ n'est pas $1$, on peut trouver $x'\in x+\Qp(1)$ qui est fixe par $\gamma$,
ce qui implique $x'\in K_\infty$ et donc $x'=0$ puisque $\varphi(x')=p x'$).  En particulier
$(U^{H_K})^{\la}$ n'est pas dense dans $U^{H_K}$ ce qui prouve que la densit\'e
de $\Pi^{\la}$ dans $\Pi$ n'est pas assur\'ee si on sort du cadre admissible.

Maintenant,
$H^1(H_K,\Qp(1))$ est une repr\'esentation admissible de $\Gamma_K$ :
son dual est $H^1(G_K,\Zp\dcroc{\Gamma_K}[\frac{1}{p}])$, qui est un 
$\Zp\dcroc{\Gamma_K}[\frac{1}{p}]$-module libre de rang~$d=[K:\Qp]$ \`a des $\Qp$-espaces
de dimension finie pr\`es, et donc $H^1(H_K,\Qp(1))\cong {\mathcal C}^0(\Gamma_K,\Qp)^d$ \`a des
$\Qp$-espaces
de dimension finie pr\`es.  Cette description prouve que $H^1(H_K,\Qp(1))^{\la}$ contient des vecteurs
non localement constants, contrairement \`a $\hat K_\infty$.
L'exactitude du foncteur
$W\mapsto W^{\la}$ peut donc \^etre mise en d\'efaut si on se permet des repr\'esentations
non admissibles.

\section{Calcul de $\hat{K}_\infty^{\la}$ dans le cas Lubin-Tate}
\label{calanlt}

Dans ce chapitre, nous consid\'erons le cas o\`u $K_\infty$ est engendr\'e par les points de torsion d'un groupe de Lubin-Tate.
\subsection{Extensions de Lubin-Tate}
\label{ltper}

Soit $F \subset K$ tel que l'extension $F/\Qp$ est galoisienne et soit $h=[F:\Qp]$. Soit $\emb$ l'ensemble des plongements de $F$ dans $\Qpbar$.  Soit $\LT$ un $\OO_F$-module formel associ\'e \`a une uniformisante $\pi_F$ de $\OO_F$ et soit $F_\infty$ l'extension de $F$ engendr\'ee par les points de $\pi_F^n$-torsion de $\LT$ pour $n \geq 1$. Soient $H_F = \Gal(\Qpbar/F_\infty)$ et  $\Gamma_F = \Gal(F_\infty / F)$. Par la th\'eorie de Lubin-Tate  (le th.~2 de \cite{LT65}), $\Gamma_F$ est isomorphe \`a $\OO_F^\times$, via le caract\`ere de  Lubin-Tate $\chi_F : \Gamma_F \to \OO_F^\times$. Si $\tau \in \emb$, soit $\chi_F^\tau = \tau \circ \chi_F$.

\begin{theo}
\label{peremb}
Si $\tau \neq \Id$, alors il existe un \'el\'ement $u_{\tau} \in \hat{F}_\infty^\times$ tel que $g(u_\tau) = \chi_F^\tau(g) \cdot u_\tau$ si $g \in \Gamma_F$. Si $\tau = \Id$, alors il n'existe pas de tel \'el\'ement.
\end{theo}

\begin{proof}
Voir le \S\,3.2 de \cite{LF09} pour le cas $\tau \neq \Id$ et le \S\,3.4 de ibid.\ pour $\tau=\Id$ (remarquons ceci dit que ces r\'esultats remontent \`a Tate, voir le \S\,4 de \cite{T67}).
\end{proof}

Soit $G_n = 1 + p^n \OO_F$ pour $n \geq 1$. On pose $F_n = F_\infty^{G_n}$ (de sorte que $F_n=F(\LT[\pi_F^{en}])$ o\`u $e$ est l'indice de ramification absolu de $F$). 
La fonction $\log$ donne lieu \`a un isomorphisme analytique de groupes
 $\log : G_n \to p^n \OO_F$ pour $n \geq 1$. Si $g \in G_n$, soit $\ell(g) = \log(g)$. 
Si $\tau \in \emb$, alors on dispose de la {\og d\'erivation dans la direction $\tau$ \fg}
 qui est un \'el\'ement $\nabla_\tau \in F \otimes_{\Qp} \Lie(\Gamma_F)$. On peut le
 construire de la mani\`ere suivante (d'apr\`es le \S\,3.1 de \cite{DI12}). Si $W$ est
 une $F$-repr\'esentation de $G_n$ et si $w \in W^{\la}$, alors il existe $m \gg 0$ 
et des \'el\'ements $\{w_\kbf \}_{\kbf \in \NN^{\emb}}$ tels que si $g \in G_m$, alors 
$g (w) = \sum_{\kbf \in \NN^{\emb}} \ell(g)^\kbf w_\kbf$, o\`u $\ell(g)^\kbf = \prod_{\tau \in \emb} \tau \circ \ell(g)^{k_\tau}$. 
On pose alors $\nabla_\tau(w) = w_{\mathbf{1}_\tau}$, la d\'eriv\'ee de $w$ dans la 
direction du plongement $\tau$. Si $\kbf\in\NN^{\emb}$, et si on pose 
$\nabla^\kbf(w) = \prod_{\tau \in \emb} \nabla_\tau^{k_\tau} (w)$, alors $w_\kbf = \nabla^\kbf(w)/\kbf!$.

On dit que $w \in W^{\la}$ est {\it $F$-analytique}
 si l'on a $\nabla_\tau(w) = 0$ pour tout $\tau\neq \Id $. Ceci \'equivaut \`a
l'existence d'une suite $\{ w_k\}_{k \geq 0}$ telle que $g(w) = \sum_{k \geq 0} \ell(g)^k w_k$ pour $g$ assez proche de $1$. On note $W^{\fan}$ les vecteurs $F$-analytiques de $W$.

\subsection{Vecteurs localement $\Qp$-analytiques}
\label{ltcalc}

Le corps $\hat{F}_\infty^{\la}$ est un sous-corps de $\hat{F}_\infty$ qui contient $F_\infty$. Si $F \neq \Qp$, alors les \'el\'ements $u_\tau$ du th.~\ref{peremb} sont des exemples d'\'el\'ements de $\hat{F}_\infty^{\la}$ qui sont  localement $\Qp$-analytiques mais pas localement constants. Rappelons que le choix de $\log(p) \in \Qp$ d\'etermine un morphisme $\Gal(\Qpbar/\Qp)$-\'equivariant $\log : \Cp^\times \to \Cp$. Soit alors $x_\tau = \log (u_\tau)$, de telle sorte que $g(x_\tau) = x_\tau + \log \chi_F^\tau(g)$ si $g \in \Gamma_F$. Pour tout plongement $\tau \in \emb \setminus \{ \Id \}$ et $n \geq 1$, soit $x_{n,\tau}$ un \'el\'ement de $F_\infty$ tel que $\| x_\tau -x_{n,\tau} \| \leq p^{-n}$. Par le lem.~\ref{descnm}, il existe $r(n) \geq 1$ tel que  $x_{n,\tau} \in F_{r(n)}$ et tel que si $m \geq r(n)$, alors $x \in \hat{F}_\infty^{G_m\dan}$ et $\| x_\tau -x_{n,\tau} \|_{G_m} = \| x_\tau -x_{n,\tau} \|$.  On peut supposer que la suite $\{r(n)\}_{n \geq 1}$ est croissante. Soit $K_n = K \cdot F_n$. Si $m \geq r(n)$ et $\{a_\kbf\}_{\kbf \in \NN^{\emb \setminus \{ \Id \} }}$ est une suite d'\'el\'ements de $K_m$ telle que $p^{n|\kbf|} a_\kbf \to 0$ quand $|\kbf| \to +\infty$, alors $a_\kbf (\xbf-\xbf_n)^\kbf \to 0$ dans $\hat{K}_\infty^{G_m\dan}$ et donc la s\'erie $\sum_{\kbf \in \NN^{\emb \setminus \{ \Id \} }} a_\kbf (\xbf-\xbf_n)^\kbf$ converge vers un \'el\'ement de $\hat{K}_\infty^{G_m\dan}$. Notons $K_m \dacc{\xbf-\xbf_n}_n$ l'ensemble des sommes de ces s\'eries, de sorte que $K_m \dacc{\xbf-\xbf_n}_n \subset \hat{K}_\infty^{G_m\dan} \subset \hat{K}_\infty^{\la}$. On a aussi une inclusion $K_{r(n)} \dacc{\xbf-\xbf_n}_n \subset K_{r(n+1)} \dacc{\xbf-\xbf_{n+1}}_{n+1}$.

\begin{theo}\label{qpanfinf}
L'application $\cup_{n \geq 1} K_{r(n)} \dacc{\xbf-\xbf_n}_n \to \hat{K}_\infty^{\la}$ est un isomorphisme d'espaces LB.
\end{theo}

\begin{coro}\label{finfan}
Si $x \in \hat{K}_\infty^{\la}$, alors $\nabla_\Id(x)=0$, et donc $\hat{K}_\infty^{\fan} = K_\infty$.
\end{coro}

\begin{rema}\label{locancorps}
{\rm (i)}
  La description de $\hat{K}_\infty^{\la}$ fournie par le th.~\ref{qpanfinf} est analogue \`a ce qui se passe pour l'anneau local d'un point de Berkovich de type IV.

{\rm (ii)}
En utilisant une version forte du th\'eor\`eme d'Ax-Sen-Tate (si $L$ est un sous-corps de $\Qpbar$,
et si $x\in\Cp$ v\'erifie $v_p(g(x)-x)\geq N$, pour tout $g\in G_L$, alors il existe
$a\in L$ tel que $v_p(x-a)\geq k-1$), on peut montrer que
$r(n)=n+1$.
\end{rema}

Remarquons pour commencer que le corps $K_\infty$ est un $F_\infty$-espace vectoriel de dimension finie $r$, et il existe donc $n \gg 0$ et des \'el\'ements $k_1,\hdots,k_r$ de $K_n$, tels que $K_\infty = \oplus_{i=1}^r F_\infty \cdot k_i$. On a alors $\hat{K}_\infty = \oplus_{i=1}^r \hat{F}_\infty \cdot k_i$ et donc $\hat{K}^{\la}_\infty = \oplus_{i=1}^r \hat{F}^{\la}_\infty \cdot k_i$ par la prop.~\ref{basan}. Afin de d\'ecrire $\hat{K}^{\la}_\infty$, il est donc suffisant de d\'eterminer $\hat{F}^{\la}_\infty$. Nous montrons donc le th.~\ref{qpanfinf} pour $K=F$. Avant cela, \'etablissons quelques r\'esultats pr\'eliminaires. 

Soit $W$ une repr\'esentation de $\Gamma_F$. Si $n \geq1$, soit $W \dacc{T}_n$ l'espace vectoriel des s\'eries formelles $\sum_{k \geq 0} a_k T^k$ avec $a_k \in W$, et $p^{nk} a_k \to 0$ quand $k \to +\infty$. On munit $W \dacc{T}_n$ d'une action de $G_n$ par la formule $g(T) = T + \ell(g)$, et en faisant agir $G_n$ sur les coefficients. Comme $\ell(g) \in p^n \OO_F$ si $g \in G_n$, cette action est bien d\'efinie. 

Si $w \in W$ est un vecteur $F$-analytique sur $G_n$ avec $n \geq 1$, alors il existe une suite $\{w_k\}_{k \geq 0}$ de $W$, avec $p^{nk}w_k \to 0$ quand $k \to +\infty$, telle que $g(w) = \sum_{k \geq 0} \ell(g)^k w_k$ si $g \in G_n$. Soit $C(w)$ l'\'el\'ement de $W \dacc{T}_n$ donn\'e par $C(w) = \sum_{k \geq 0} (-1)^k w_k T^k$.

\begin{lemm}\label{fanbsenf}
Si $w \in W$ est $F$-analytique sur $G_n$, alors $C(w) \in W \dacc{T}_n^{G_n}$.
\end{lemm}

\begin{proof}
Si $g \in G_n$, alors $g(C(w))=\sum_{k \geq 0} (-1)^k g(w_k) (T+\ell(g))^k$. Comme $w$ est $F$-analytique sur $G_n$, $w_k$ l'est aussi et on a $g(w_k) = \sum_{j \geq 0} \ell(g)^j w_{j+k} \binom{k+j}{j}$ avec $w_i = \nabla^i(w)/i!$ pour $i \geq 0$. On a  donc 
\begin{align*} 
g(C(w)) & =\sum_{k \geq 0} (-1)^k g(w_k) (T+\ell(g))^k \\
& = \sum_{k \geq 0} (-1)^k \sum_{j \geq 0} \ell(g)^j w_{j+k} \binom{k+j}{j} \sum_{i+\ell = k} T^i \ell(g)^{\ell} \binom{k}{i} \\
& = \sum_{i \geq 0} (-1)^i T^i  \sum_{m \geq i} \ell(g)^{m-i} w_m \binom{m}{i} \sum_{j+\ell = m - i} (-1)^\ell  \binom{m-i}{\ell} 
\end{align*}
Comme $\sum_{j+\ell = m - i} (-1)^\ell  \binom{m-i}{\ell}  = (1-1)^{m-i} = 0$ sauf si $m=i$, on a finalement $g(C(w)) = \sum_{i \geq 0} (-1)^i T^i  w_i  = C(w)$.
\end{proof}

\begin{lemm}\label{approz}
Si $n \geq 1$ et $m \geq r(n)$, alors il existe $x_n \in p^n \OO_{\hat{F}_\infty}$ tel que $g(x_n)=x_n+ \sum_{\tau \neq \Id} \log \chi_F^{\tau} (g)$ si $g \in G_{n+m}$. 
\end{lemm}

\begin{proof}
Si $x = \sum_{\tau \neq \Id} x_\tau$, alors $g(x)=x + \sum_{\tau \neq \Id} \log \chi_F^{\tau} (g)$ pour  $g \in G_n$ avec $n \geq 1$. Il suffit alors de prendre $x_n = x - \sum_{\tau \neq \Id} x_{n+m,\tau}$ pour $m \geq r(n)$.
\end{proof}

Soit $\chi : \Gamma_F \to \Zp^\times$ le caract\`ere $g \mapsto \chi(g) = \prod_{\tau \in \emb} \chi_F^\tau(g) = \Nm_{F/\Qp} (\chi_F(g))$. Soit $\hat{F}_\infty \dacc{U}_m$ l'anneau des s\'eries formelles $\sum_{k \geq 0} a_k U^k$ avec $a_k \in \hat{F}_\infty$, et $p^{mk} a_k \to 0$ quand $k \to +\infty$. On d\'efinit une action de $G_m$ sur $\hat{F}_\infty \dacc{U}_m$ par $g(U) = U + \log \chi(g)$.

\begin{prop}\label{bseninv}
Si $m \geq 1$, alors $\hat{F}_\infty \dacc{U}_m^{G_m} = F_m$.
\end{prop}

\begin{proof}
Soit $\bsen^m$ l'anneau dont la construction est rappel\'ee au \S\,\ref{bsensec}. 
L'anneau $\hat{F}_\infty \dacc{U}_m$ admet de mani\`ere \'evidente une injection  $\Gal(\Qpbar/F_m)$-\'equivariante dans $\bsen^m$ et la proposition suit du th.~\ref{pcsenth}.
\end{proof}

\begin{coro}\label{finfanq}
L'ensemble des vecteurs $F$-analytiques de $\hat{F}_\infty^{G_n\dan}$ est $F_n$.
\end{coro}

\begin{proof}
Soit $x \in \hat{F}_\infty$ un vecteur $F$-analytique sur $G_n$. Son image $C(x)$ dans $\hat{F}_\infty \dacc{T}_n$ est fix\'ee par $G_n$ par le lem.\ref{fanbsenf} appliqu\'e \`a $W=\hat{F}_\infty$. Soit $x_n$ comme dans le lem.\ref{approz}. L'application $\hat{F}_\infty \dacc{T}_n \to \hat{F}_\infty \dacc{U}_n$ donn\'ee par $T \mapsto U - x_n$ est bien d\'efinie comme $x_n \in p^n \OO_{\hat{F}_\infty}$, et c'est une bijection $G_{n+m}$-\'equivariante. En appliquant la prop.~\ref{bseninv} \`a l'image de $C(x)$ dans $\hat{F}_\infty \dacc{U}_{m+n}$, on trouve que $x \in F_{n+m}$. Un vecteur analytique sur $G_n$ et constant sur $G_{n+m}$ est constant sur $G_n$, et donc $x \in F_n$.
\end{proof}

\begin{proof}[D\'emonstration du th.~\ref{qpanfinf}]
Soit $z$ un \'el\'ement de $\hat{F}_\infty^{G_\ell\dan}$ avec $\ell \geq 1$, et $z_\kbf=\nabla^\kbf(z)/\kbf!$ si $\kbf \in \NN^{\emb}$. Par les lem.~\ref{liecont} et \ref{descnm}, on a $z_\kbf \in \hat{F}_\infty^{G_\ell\dan}$ pour tout $\kbf \in \NN^{\emb}$ et il existe une constante $n$ telle que $\| z_\kbf \|_{G_m} \leq p^{(n-1) |\kbf|} \| z \|_{G_\ell}$ quel que soit $m \geq \ell$. Si $m \geq \max(r(n),\ell)$ et $\ibf \in \NN^{\emb \setminus \{ \Id \}}$, alors la s\'erie
\[ y_\ibf = \sum_{\kbf \in \NN^{\emb \setminus \{ \Id \} }} (-1)^{|\kbf|} (\xbf-\xbf_n)^\kbf z_{\kbf+\ibf} \binom{\kbf+\ibf}{\kbf} \]
converge dans $\hat{F}_\infty^{G_m\dan}$. Comme $g(x_\tau-x_{n,\tau})=(x_\tau-x_{n,\tau})+\tau \circ \ell(g)$ pour $g \in G_m$, on a 
\[ \nabla_\tau (\xbf-\xbf_n)^\kbf  = \begin{cases} k_\tau (\xbf-\xbf_n)^{\kbf-\mathbf{1}_\tau} & \text{si $k_\tau \geq 1$,} \\
0 & \text{si $k_\tau=0$}. \end{cases} \]
Cette formule, et le fait que
\[ y_\ibf = \sum_{\kbf \in \NN^{\emb \setminus \{ \Id \} }} (-1)^{|\kbf|} (\xbf-\xbf_n)^\kbf z_{\kbf+\ibf} \binom{\kbf+\ibf}{\kbf} = \frac{1}{\ibf!} \sum_{\kbf \in \NN^{\emb \setminus \{ \Id \} }} (-1)^{|\kbf|} (\xbf-\xbf_n)^\kbf 
\frac{\nabla^\kbf(\nabla^\ibf(z))}{\kbf!}, \]
impliquent que $\nabla_\tau(y_\ibf)=0$ pour tout $\tau \neq \Id$ et $\ibf \in \NN^{\emb \setminus \{ \Id \} }$. Les \'el\'ements $y_\ibf$ sont donc des vecteurs localement $F$-analytiques de $\hat{F}_\infty^{G_m\dan}$. Ils appartiennent \`a $F_m$ par le cor.~\ref{finfanq}. 

La formule d\'efinissant $y_\ibf$ et le fait que $\|\xbf-\xbf_n\|_{G_m} \leq p^{-n}$ impliquent que l'on a $\| y_\ibf \|_{G_m} \leq p^{(n-1)|\ibf|} \|z\|_{G_\ell}$, et donc que la s\'erie $\sum_{\ibf \in \NN^{\emb \setminus \{ \Id \} }} y_\ibf (\xbf-\xbf_n)^\ibf$ converge  dans $\hat{F}_\infty^{G_m\dan}$. On a alors
\begin{align*}
\sum_{\ibf \in \NN^{\emb \setminus \{ \Id \} }} y_\ibf (\xbf-\xbf_n)^\ibf  & = \sum_{\ibf \in \NN^{\emb \setminus \{ \Id \} }}  \sum_{\kbf \in \NN^{\emb \setminus \{ \Id \} }} (-1)^{|\kbf|} (\xbf-\xbf_n)^\kbf z_{\kbf+\ibf} \binom{\kbf+\ibf}{\kbf} (\xbf-\xbf_n)^\ibf  \\
& = \sum_{\jbf \in \NN^{\emb \setminus \{ \Id \} }} z_\jbf (\xbf-\xbf_n)^\jbf \sum_{\kbf+\ibf=\jbf}  (-1)^{|\kbf|} \binom{\jbf}{\kbf}  \\
& = z_0.
\end{align*}
On en d\'eduit que $z = \sum_{\ibf \in \NN^{\emb \setminus \{ \Id \} }} y_\ibf (\xbf-\xbf_n)^\ibf$, et donc que $z \in F_m \dacc{\xbf-\xbf_n}_n$. 

L'application $\cup_{n \geq 1} F_{r(n)} \dacc{\xbf-\xbf_n}_n \to \hat{F}_\infty^{\la}$ est une bijection continue entre espaces LB, et donc un isomorphisme d'espaces LB par le th\'eor\`eme de l'image ouverte.
\end{proof}

\subsection{G\'en\'eralisation \`a $\bdr$}
\label{bdran}

Dans ce {\S}, nous montrons une g\'en\'eralisation du th.~\ref{qpanfinf}, o\`u $\hat{K}_\infty$ est remplac\'e par $(\bdr^+/t^k)^{H_K}$ avec $k \geq 1$.

\begin{lemm}\label{cvthbdr}
Soit $E$ une extension finie de $\Qp$ et $f(T) = \sum_{k \geq 0} a_k T^k \in E\dcroc{T}$. Si $x \in \bdr^+/t^k$ avec $k \geq 1$, alors la s\'erie $f(x)$ converge dans $\bdr^+ / t^k$ si et seulement si la s\'erie  $f(\theta(x))$ converge dans $\Cp$.
\end{lemm}

\begin{proof}
Rappelons que $\bdr^+/t^k$ est un espace de Banach, la boule unit\'e \'etant l'image de $\atplus \to \bdr^+ / t^k$. On peut agrandir $E$ de telle sorte qu'il contienne un \'el\'ement de valuation $\vp(\theta(x))$ et il suffit alors de montrer que si $\theta(x) \in \OO_{\Cp}$, alors la suite $\{x^n\}_{n \geq 0}$ est born\'ee dans $\bdr^+/t^k$. Soit $x_0$ un \'el\'ement de $\atplus$ tel que $\theta(x) = \theta(x_0)$. On peut \'ecrire $x=x_0+([\tilde{p}]-p) y + t^k z$ o\`u $y \in \btplus$ et $z \in \bdr^+$. On a alors
\[ x^n = x_0^n + \binom{n}{1} x_0^{n-1} ([\tilde{p}]-p) y + \cdots + \binom{n}{k-1} x_0^{n-(k-1)} (([\tilde{p}]-p) y)^{k-1} + t^k z_k, \]
avec $z_k \in \bdr^+$, et donc $x^n \in (\atplus + y \atplus + \cdots + y^{k-1} \atplus) + t^k \bdr^+$ pour tout $n \geq 0$.
\end{proof}

Rappelons que $\Qpbar \subset \bdr^+$. Si $y \in (\bdr^+)^\times$, il existe $y_0 \in \Qpbar$ tel que $\theta(y/y_0) \in 1+\MM_{\Cp}$. On d\'efinit  une application $\log : (\bdr^+)^\times \to \bdr^+$ par $\log(y) = \log(y_0) + \log(y/y_0)$, o\`u $\log(y/y_0) = \sum_{n \geq 1} (-1)^{n-1} (y/y_0-1)^n/n$ converge dans $\bdr^+$ par le lem.~\ref{cvthbdr}. Cette application ne d\'epend pas du choix de $y_0$, et est $\Gal(\Qpbar/\Qp)$-\'equivariante. Si $\tau \neq \Id$, un analogue du th.~\ref{peremb} nous donne des \'el\'ements $v_\tau \in (\bdr^+)^\times$ tels que $g(v_\tau) = \chi_F^\tau(g)  \cdot v_\tau$, et les \'el\'ements $\log(v_\tau) \in \bdr^+$ satisfont $g(\log(v_\tau)) = \log(v_\tau) + \tau \circ \ell(g)$. \'Ecrivons $x'_\tau = \log(v_\tau)$ de telle sorte que $\theta(x'_\tau)=x_\tau$. Soit $K_{r(n)} \dacc{\Tbf}_n$ l'espace des s\'eries formelles $\sum_{\kbf \in \NN^{\emb \setminus \{ \Id \} }} a_\kbf \Tbf^\kbf$ o\`u $a_\kbf \in K_{r(n)}$ et $a_\kbf p^{n|\kbf|} \to 0$ quand $|\kbf| \to + \infty$. 

\begin{lemm}\label{sumconv}
Si $f(T) \in K_{r(n)} \dacc{\Tbf}_n$, alors $f(\xbf'-\xbf_n)$ converge dans $(\bdr^+/t^k)^{H_K}$.
\end{lemm}

\begin{proof}
Cela suit du lem.~\ref{cvthbdr}.
\end{proof}

Soit $t_F$ un \'el\'ement de $\bdr^+ \setminus \{0\}$ tel que $g(t_F) = \chi_F(g) t_F$. Rappelons que $t_F/t$ est une unit\'e de $\bdr^+$. L'\'el\'ement $t_F$ est un vecteur $F$-analytique, alors que $t$ est seulement $\Qp$-analytique, c'est pourquoi nous pr\'ef\'erons utiliser $t_F$ dans l'\'enonc\'e du th.~\ref{qpanbdr}.

\begin{theo}\label{qpanbdr}
Si $k \geq 1$, alors l'application $\oplus_{i=0}^{k-1} (\cup_{n \geq 1}  K_{r(n)} \dacc{\xbf'-\xbf_n}_n) \cdot t_F^i \to (\bdr^+/t^k)^{H_K}_{\la}$ est un isomorphisme d'espaces LB.
\end{theo}

\begin{proof}
Le th.~\ref{qpanfinf} correspond au cas $k=1$. Supposons que le th\'eor\`eme est vrai pour $k-1$, et soit $y \in (\bdr^+/t^k)^{H_K}_{\la}$. Par le th.~\ref{qpanfinf}, on a $\theta(y) =f(\xbf-\xbf_n)$ pour un $n \gg 0$ avec $f(\Tbf) \in  K_{r(n)} \dacc{\Tbf}_n$. Par le lem.~\ref{sumconv}, la s\'erie $f(\xbf'-\xbf_n)$ converge  dans $\bdr^+/t^k$. On a $y-f(\xbf'-\xbf_n)  \in t_F \cdot \bdr^+ /t^k$, et comme $t_F$ est lui-m\^eme un vecteur $\Qp$-analytique, on peut \'ecrire $y = f(\xbf'-\xbf_n) + t_F \cdot z$ avec $z \in (\bdr^+/t^{k-1})^{H_K}_{\la}$. Ceci permet de montrer le th\'eor\`eme par r\'ecurrence sur $k$.
\end{proof}

\subsection{Vecteurs localement $F$-analytiques}
\label{flocan}

Nous supposons toujours que $K_\infty$ est engendr\'e par les points de torsion d'un $\OO_F$-module formel. Soit $W^{\fan}$ l'espace des vecteurs localement $F$-analytiques d'une $\hat{K}_\infty$-repr\'esentation semi-lin\'eaire $W$ de dimension finie de $\Gamma_K$. Rappelons que $W\dacc{T}_n$ est l'espace vectoriel des s\'eries $\sum_{k \geq 0} w_k T^k$ o\`u $p^{nk} w_k \to 0$ dans $W$. Cet espace est muni d'une action de $G_n$ comme au d\'ebut du \S\,\ref{ltcalc}.

\begin{lemm}
\label{nplinj}
L'application $K_{n+k} \otimes_{K_n} W\dacc{T}_n^{G_n} \to W\dacc{T}_{n+k}^{G_{n+k}}$ est injective.
\end{lemm}

\begin{proof}
Si $\sum_{i=1}^r a_i w_i$ est un \'el\'ement de longueur minimale ayant pour image $0$, alors il en est de m\^eme pour $\sum_{i=1}^r g(a_i) w_i$ si $g \in G_n$ de telle sorte que $a_i \in K_n \cdot a_1$, et l'application est bien injective.
\end{proof}

\begin{coro}
\label{wfanplin}
L'application $K_{n+k} \otimes_{K_n} W^{G_n\dan,\fan} \to W^{G_{n+k}\dan,\fan}$ est injective.
\end{coro}

\begin{proof}
L'application $C : W^{G_n\dan,\fan} \to W\dacc{T}_n$ donn\'ee comme pr\'ec\'edemment par $x \mapsto \sum_{k \geq 0} (-1)^k x_k T^k$ envoie $W^{G_n\dan,\fan}$ dans $W\dacc{T}_n^{G_n}$ par le lem.~\ref{fanbsenf}. Elle est manifestement injective, et le corollaire suit du lem.~\ref{nplinj}.
\end{proof}

\begin{theo}
\label{fanvec}
Si $W$ est une $\hat{K}_\infty$-repr\'esentation semi-lin\'eaire de dimension finie de $\Gamma_K$, alors l'application $\hat{K}_\infty^{\la} \otimes_{K_\infty} W^{\fan} \to W^{\la}$ est un isomorphisme.
\end{theo}

\begin{proof}
Montrons d'abord que $\hat{K}_\infty^{\la} \otimes_{K_\infty} W^{\fan} \to W^{\la}$ est injective. Si elle ne l'est pas, alors soit $\sum_{i=1}^r \alpha_i x_i = 0$ une relation non triviale de longueur minimale, avec $\alpha_i \in \hat{K}_\infty^{\la}$ et $x_i \in W^{\fan}$. On peut supposer que $\alpha_1=1$. Si $\tau \in \emb \setminus \{ \Id \}$, alors $\nabla_\tau(x_i) = 0$ et donc $\sum_{i=2}^r \nabla_\tau(\alpha_i) x_i = 0$. Cette relation \'etant plus courte, on a $\nabla_\tau(\alpha_i) = 0$ pour tout $\tau \in \emb \setminus \{ \Id \}$, et donc $\alpha_i$ appartient \`a $\hat{K}_\infty^{\fan}  = K_\infty$. La relation \'etait donc triviale.

Montrons \`a pr\'esent que $\hat{K}_\infty^{\la} \otimes_{K_\infty} W^{\fan} \to W^{\la}$ est surjective. L'injectivit\'e implique que $\dim_{K_\infty} W^{\fan}$ est de dimension finie et donc, par le cor.~\ref{wfanplin}, que l'application $K_\infty \otimes_{K_m} W^{G_m\dan,\fan} \to W^{\fan}$ est un isomorphisme si $m \gg 0$.  Si $z \in W^{G_\ell\dan}$ pour $\ell \geq 1$, alors les m\^emes arguments que dans la preuve du th.~\ref{qpanfinf} montrent qu'il existe $n$ et $m \geq \max(r(n),\ell)$ tels que les s\'eries
\[ y_\ibf = \sum_{\kbf \in \NN^{\emb \setminus \{ \Id \} }} (-1)^{|\kbf|} (\xbf-\xbf_n)^\kbf z_{\kbf+\ibf} \binom{\kbf+\ibf}{\kbf} \]
convergent dans $W^{G_m\dan}$ pour tout $\ibf$,  que $y_\ibf$ est localement $F$-analytique, et que l'on a $z = \sum_{\ibf \in \NN^{\emb \setminus \{ \Id \} }} y_\ibf (\xbf-\xbf_n)^\ibf$ dans $W^{G_m\dan}$. Comme $W^{G_m\dan,\fan}$ est de dimension finie sur $K_m$, ceci implique que $z \in \hat{K}_\infty^{G_m\dan} \otimes_{K_m} W^{G_m\dan,\fan}$ et la surjectivit\'e en d\'ecoule.
\end{proof}

Pour terminer, remarquons que $W^{\fan}$ peut \^etre construit \`a partir de la th\'eorie de Sen classique associ\'ee au caract\`ere $\chi = \Nm_{F/\Qp}(\chi_F) : \Gamma_K \to \Zp^\times$. Plus pr\'ecis\'ement, il existe un isomorphisme $S : W^{\fan} \to K_\infty \otimes_{K_\infty^{\chi}} \dsen^{\chi}(W)$, qui v\'erifie $S \circ \nabla_{\Id} = \Theta_{\chi} \circ S$. On a   $S(x) = \exp(\alpha \nabla_{\Id})(x)$ o\`u $\alpha \in \pi_F^n \OO_{\hat{K}_\infty}$ est tel que $g(\alpha)-\alpha = \log \chi(g)/\chi_F(g)$ pour $g \in \Gamma_m$ avec $m$ et $n \gg 0$.

\subsection{Le cas de l'extension de Kummer}

Si $n\geq 1$, soient $\omega\in K^*$ pas une racine de l'unit\'e,
$\omega_n = \omega^{1/p^n}$ et $K_n =  K(\omega_n,\zeta_{p^n})$,
 et soit $K_\infty = \cup_{n \geq 1} K_n$. Si $\Gamma_K=\Gal(K_\infty/K)$,
on a une suite exacte $0 \to \Zp \to \Gamma_K \to \Zp^\times$, 
l'image de la fl\`eche de droite \'etant un sous-groupe ouvert de $\Zp^\times$. 
Soit $\tau \in \Gal(K_\infty / K(\zeta_{p^\infty}))$ un g\'en\'erateur topologique. 
Si $g \mapsto c(g)$ d\'enote le cocycle de Kummer associ\'e \`a $\omega$,
 alors $K_\infty$ est le noyau de 
$g \mapsto \smat{\chi(g) & c(g) \\ 0 & 1}$. Comme $H^1(G_K,\Cp(1)) = \{ 0 \}$, il existe $\alpha \in \Cp$ 
tel que $c(g) = g(\alpha)\chi(g)-\alpha$. Ceci implique que $g(\alpha) = \alpha/\chi(g) + c(g)/\chi(g)$, 
et donc que $\alpha \in \hat{K}_\infty^{\la}$. Dans des notations analogues 
\`a celles du \S\,\ref{ltper}, on a le r\'esultat suivant.

\begin{prop}\label{kumlocan}
On a $\hat{K}_\infty^{\la} = \cup_{n \geq 1} K_{r(n)} \dacc{ \alpha - \alpha_n }_n$.
\end{prop}

\begin{proof}
Donnons une id\'ee rapide de la d\'emonstration. 
Soit $x\in \hat{K}_\infty^{G_n\dan}$ et soit $\nabla_\tau$ l'op\'erateur 
diff\'erentiel associ\'e \`a $\tau$.
Soit $y_i = \sum_{k \geq 0} (-1)^k (\alpha - \alpha_m)^k \nabla_\tau^{k+i}(x) \binom{k+i}{k}$ 
comme dans la preuve du th.~\ref{qpanfinf}. 
Les m\^emes arguments montrent qu'il existe $m \geq n$ tel que $y_i \in \hat{K}_\infty^{G_m\dan}$ 
pour tout $i$, et que $x = \sum_{i \geq 0} y_i (\alpha - \alpha_m)^i$ dans $\hat{K}_\infty^{G_m\dan}$. 
On a $\nabla_\tau(y_i) = 0$ et 
donc $y_i \in \widehat{K_m(\zeta_{p^\infty})}^{\la}$. On en d\'eduit que $y_i \in K_m$, et
le r\'esultat.
\end{proof}

\section{Calcul de $\hat{K}_\infty^{\fin}$ et $\hat{K}_\infty^{\la}$ dans le cas $\SL_2$}
\label{sl2sec}

Dans ce chapitre, on s'int\'eresse \`a pr\'esent au cas o\`u $\Gamma_K$ est isomorphe \`a un sous-groupe ouvert de $\SL_2(\Zp)$, de sorte que $\Lie(\Gamma_K) = \LieSL$, un $\Qp$-espace vectoriel de dimension $3$.

\subsection{Vecteurs finis et poids de Hodge-Tate}
\label{finsl}

Dans tout ce chapitre, par {\og repr\'esentation \fg} de $\LieSL$, on entend repr\'esentation $\Qp$-lin\'eaire de dimension finie.  Le r\'esultat suivant est classique.

\begin{prop}\label{repirrsl}
Soit $V$ la repr\'esentation standard de $\LieSL$.

{\rm (i)}
 Si $X$ est une repr\'esentation irr\'eductible de $\LieSL$, alors $X = \Sym^k V$ avec $k \geq 0$;

{\rm (ii)}
 Toute repr\'esentation de $\LieSL$ est somme directe de repr\'esentations irr\'eductibles;
\end{prop}

\begin{proof}
Le (ii) est  dans le \S\,6.2 de \cite{GRALG1}, le (i) dans le \S\,1.3 de \cite{GRALG8}.
\end{proof}

Si $X$ est une repr\'esentation de $\Gamma_K$, alors par la th\'eorie de Lie, $X$ est aussi muni d'une action compatible de $\LieSL$. L'isomorphisme entre $\Gamma_K$ et un sous-groupe ouvert de $\SL_2(\Zp)$ se traduit par l'existence d'un $\Qp$-espace vectoriel $V$ de dimension $2$, sur lequel $\Gamma_K$ agit par des automorphismes de d\'eterminant $1$. L'espace $V$ muni de l'action correspondante de $\LieSL$ est alors la repr\'esentation standard de $\LieSL$ comme ci-dessus.

Comme on a un morphisme $G_K \to \Gamma_K$, une repr\'esentation de $\Gamma_K$ est aussi une repr\'esentation de $G_K$. Les poids de Hodge-Tate de $V$ sont $s$ et $-s$ avec $s \in \Qpbar$. On suppose que $s \neq 0$ (on obtient un exemple de telle repr\'esentation en partant d'une forme modulaire $f$ de poids $k\geq 2$,
non CM et non ordinaire en $p$, et en tordant la restriction \`a $G_{\Qp(\mu_{2p})}$
de la repr\'esentation $V_f$ associ\'ee \`a $f$ par $(\det V_f)^{-(k-1)/2}$; on a alors $s=\frac{k-1}{2}$).
 Les poids de Hodge-Tate de $\Sym^k V$ sont alors $\{ -ks,-(k-2)s, \hdots, ks \}$ si $k \geq 0$.

\begin{coro}
\label{toutalg}
Si $X$ est une repr\'esentation de $G_K$ qui se factorise par $\Gamma_K$, alors l'op\'erateur de Sen de $X$ est semisimple, \`a valeurs propres dans $s \cdot \ZZ$.
\end{coro}

\begin{proof}
Cela suit des remarques pr\'ec\'edentes et de la prop.~\ref{repirrsl}.
\end{proof}

\begin{prop}\label{vecfinec}
Soit $W$ une $\hat{K}_\infty$-repr\'esentation semi-lin\'eaire de dimension finie de~$\Gamma_K$.

{\rm (i)}
 Si $W^{\fin} \neq \{0\}$, alors $W$ a un poids de Hodge-Tate qui appartient \`a $s \cdot \ZZ$;

{\rm (ii)}
 Si $W^{\fin}$ contient une base de $W$, alors l'op\'erateur de Sen de $W$ est semisimple, \`a valeurs propres dans $s \cdot \ZZ$.
\end{prop}

\begin{proof}
Soit $x \in W^{\fin}$. Par la prop.~\ref{repirrsl}, il existe une repr\'esentation irr\'eductible $Y(x) \subset W^{\fin}$ de $\LieSL$ telle que $Y(x)$ contient $x$. Par la prop.~\ref{repirrsl}, $Y(x) = \Sym^k V$ pour un $k \geq 0$. Il existe un sous-groupe ouvert $\Gamma_L$ de $\Gamma_K$ tel que $Y(x)$ est stable sous l'action de $\Gamma_L$, et les poids de Hodge-Tate de $Y(x)$ sont donc dans $s \cdot \ZZ$.

Le (i) r\'esulte de ce que l'on a une application non nulle $\Cp \otimes_{\Qp} Y(x) \to \Cp \otimes_{\hat{K}_\infty} W$, et le 
(ii) de ce que l'on a une application surjective $\Cp \otimes_{\Qp} W^{\fin} \to \Cp \otimes_{\hat{K}_\infty} W$ et du fait que $W^{\fin} = \cup_{x \in W^{\fin}} Y(x)$.
\end{proof}

\'Ecrivons $V=\Qp e_1 \oplus \Qp e_2$ de telle sorte que l'action de $\Gamma_K$ dans la base $(e_1,e_2)$ soit compatible avec l'isomorphisme entre $\Gamma_K$ et un sous-groupe ouvert de $\SL_2(\Zp)$. 
Les poids de Hodge-Tate de $\Sym^2 V$ sont $-s,0,s$,
et il existe donc un plongement de $\Sym^2 V$ dans~$\Cp$. Soient 
\[ x_1=e_1 \otimes e_1 \qquad x_2=e_2 \otimes e_2 \qquad y = \frac{e_1 \otimes e_2 + e_2 \otimes e_1}{2} \qquad \delta = \frac{e_1 \otimes e_2 - e_2 \otimes e_1}{2}. \]
les \'el\'ements correspondants de $\Cp$, avec $\Sym^2 V = \Qp x_1 \oplus \Qp x_2 \oplus \Qp y$ et $\det V = \Qp \delta$, ce qui fait que $\delta \in K^\times$
puisque $\det V$ est la repr\'esentation triviale, \'etant donn\'e que $\Gamma_K\subset \SL_2(\Zp)$.

\begin{prop}\label{finkinf}
On a 
\[ \hat{K}_\infty^{\fin} = \oplus_{k \geq 0} (K_\infty \otimes_{\Qp} \Sym^{2k} V)=
K_\infty[x_1,x_2,y]/(y^2-x_1x_2-\delta^2). \]
\end{prop}

\begin{proof}
Soit $x \in \hat{K}_\infty^{\fin}$. Par la prop.~\ref{repirrsl}, il existe une 
repr\'esentation irr\'eductible $Y(x) \subset \hat{K}_\infty^{\fin}$ de $\LieSL$ telle que 
$Y(x)$ contient $x$. Par la prop.~\ref{repirrsl}, $Y(x)$ est isomorphe \`a $\Sym^n V$ 
pour un $n \geq 0$. Il existe un sous-groupe ouvert $\Gamma_L$ de $\Gamma_K$ tel que $Y(x)$
 est stable sous l'action de $\Gamma_L$ et comme il existe une injection $G_L$-\'equivariante
 de $Y(x)$ dans $\Cp$, l'un des poids de Hodge-Tate de $Y(x)$ est nul, et donc $n$ est pair; 
on pose $n=2k$. Comme $\Sym^{2k} V$ n'a qu'un seul poids de Hodge-Tate nul, l'injection 
$G_L$-\'equivariante $\Sym^{2k} V \to \Cp$ est unique \`a multiplication par un \'el\'ement 
de $L^\times$ pr\`es. On a donc $x \in K_\infty \otimes_{\Qp} \Sym^{2k} V$, 
et le r\'esultat s'en d\'eduit en exprimant $\Sym^{2k} V$ en termes de
$\Sym^2 V$ et $\det V$.
\end{proof}

\subsection{Vecteurs localement analytiques}
\label{slocan}

Nous calculons \`a pr\'esent $\hat{K}_\infty^{\la}$. Pour $i=1,2$, soit $\{x_{i,n}\}_{n \geq 1}$ une suite avec $x_{i,n} \in K_\infty$ telle que $\| x_i - x_{i,n} \| \leq p^{-n}$. Soit $r(n)$ tel que $x_{i,n} \in K_{r(n)}$ et tel que si $m \geq r(n)$, alors $x_i \in \hat{K}_\infty^{G_m\dan}$ et $\| x_i - x_{i,n} \|_{G_m} = \| x_i - x_{i,n} \|$. On peut supposer que la suite $\{r(n)\}_{n \geq 1}$ est croissante. Si $m \geq r(n)$, on note $K_m \dacc{\xbf-\xbf_n}_n$ l'espace des s\'eries $\sum_{\kbf \in \NN^2} a_\kbf (\xbf-\xbf_n)^\kbf$ avec $a_\kbf \in K_m$ tels que $p^{n|\kbf|} a_\kbf \to 0$. Les \'el\'ements de $K_m \dacc{\xbf-\xbf_n}_n$ sont des \'el\'ements de $\hat{K}_\infty^{G_m\dan} \subset \hat{K}_\infty^{\la}$, et on a une inclusion $K_{r(n)} \dacc{\xbf-\xbf_n}_n \subset K_{r(n+1)} \dacc{\xbf-\xbf_{n+1}}_{n+1}$.

\begin{theo}\label{kinfansl}
L'application $\cup_{n \geq 1} K_{r(n)} \dacc{\xbf-\xbf_n}_n \to \hat{K}_\infty^{\la}$ est un isomorphisme d'espaces LB.
\end{theo}

\begin{exem}\label{exyansl}
On a bien $y \in \cup_{n \geq 1} K_{r(n)} \dacc{\xbf-\xbf_n}_n$, ce que l'on peut voir comme suit. On a $y^2=\delta^2+x_1x_2$. Si $\{y_n\}_{n \geq 1}$ est une suite de $K_\infty$ qui tend vers $y$, alors 
\[ y  = \pm y_n \sqrt{1+ \frac{\delta^2+((x_1-x_{1,n})+x_{1,n})((x_2-x_{2,n})+x_{2,n})-y_n^2}{y_n^2}}, \]
et il suffit de d\'evelopper et d'utiliser la formule $\sqrt{1+X} = \sum_{k \geq 0} \binom{1/2}{k} X^k$, le r\'esultat convergeant dans $K_{r(n)} \dacc{\xbf-\xbf_n}_n$ pour $n \gg 0$.
\end{exem}

Soit $L_\infty = K_\infty(\mu_{p^\infty})$ et $\Gamma_L = \Gal(L_\infty/K)$. On note $\nabla$ le g\'en\'erateur habituel de l'alg\`ebre de Lie de $\Gal(K(\mu_{p^\infty})/K)$. Le groupe $\Gamma_L$ est un groupe de Lie $p$-adique de dimension $4$, dont l'alg\`ebre de Lie est isomorphe \`a $\LieSL \oplus \Qp \nabla$. Soient $D_1 = \smat{ 0 & 0 \\ 1 & 0} \in \LieSL$ et $D_2 = \smat{ 0 & 1 \\ 0 & 0} \in \LieSL$ et $H = \smat{ -1 & 0 \\ 0 & 1} \in \LieSL$, de sorte que $[D_1,D_2] = H$. 

\begin{lemm}\label{acslx}
On a $D_1(x_1) = D_2(x_2) = 2y$ et $D_1(x_2) = D_2(x_1) = 0$.
\end{lemm}

Soient $\partial_1$, $\partial_2$ et $J : \hat{L}_\infty^{\la} \to \hat{L}_\infty^{\la}$ les op\'erateurs $\partial_i = 1/2y \cdot D_i$ et $J = x_1 D_1 - x_2 D_2 + y H$. On a $\partial_i = d/dx_i$ sur $L_\infty[x_1,x_2]$ et $J = 4y^3 \cdot [\partial_1,\partial_2]$. On peut donc voir $J$ comme un op\'erateur de {\og courbure \fg}.

\begin{lemm}\label{exisz}
Il existe $z \in \hat{L}_\infty^{\la}$ tel que $J(z)=1$ et $\nabla(z) = s \neq 0$.
\end{lemm}

\begin{proof}
La repr\'esentation $V(s)$ de $G_{K(\mu_{p^n})}$ existe pour $n \gg 0$. Elle a un poids de Hodge-Tate nul et s'envoie donc dans $\Cp$. Si $w$ d\'esigne l'image de $e_1(s)$, alors $w \in \hat{L}_\infty^{\la}$ et $J(w)=w\delta$ et $\nabla(w) = sw$. Si $z=\log(w)/\delta$,  alors $J(z)=1$ et $\nabla(z)=s$.
\end{proof}

\begin{proof}[D\'emonstration du th.~\ref{kinfansl}]
Soit $\{z_n\}_{n \geq 1}$ une suite d'\'el\'ements de $L_\infty$ telle que $z_n \to z$. Quitte \`a modifier la d\'efinition de $r(n)$, on peut supposer que $z_n \in L_{r(n)}$ et que si $m \geq r(n)$, alors $z \in \hat{L}_\infty^{G_m\dan}$ et $\|z-z_n\|_{G_m}=\|z-z_n\|$. Si $y \in \hat{L}_\infty^{G_q\dan}$ pour un $q \geq 1$, alors par les lemmes \ref{liecont} et \ref{descnm}, il existe une constante $h$ telle que $\| J^k(y)/k! \|_{G_\ell} \leq p^{(h-1) k} \| y \|_{G_q}$ quels que soient $\ell \geq q$ et $k \geq 0$. Si $\ell \geq \max(q,r(h))$ et $i \geq 0$, alors la s\'erie 
\[ c_i(y) = \sum_{k \geq 0} (-1)^k (z-z_h)^k \cdot \frac{J^{k+i}(y)}{(k+i)!} \binom{k+i}{i} \]
converge dans $\hat{L}_\infty^{G_\ell\dan}$ vers un \'el\'ement qui satisfait $\| c_i(y) \|_{G_\ell} \leq p^{(h-1)i} \| y \|_{G_q}$. On a donc $y = \sum_{i \geq 0} c_i(y) (z-z_h)^i$ dans $\hat{L}_\infty^{G_\ell\dan}$, et de plus $J(c_i(y)) = 0$ si $i \geq 0$.

Si $c \in \hat{L}_\infty^{G_\ell \dan}$ v\'erifie $J(c)=0$, alors $\partial_1 \partial_2 (c) = \partial_2 \partial_1 (c)$. Par les lemmes \ref{liecont} et \ref{descnm}, il existe une constante $n$ telle que $\| \partial^\kbf(c) / \kbf! \|_{G_m} \leq p^{(n-1) |\kbf| } \| c \|_{G_\ell}$ pour $\kbf \in \NN^2$ et $m\geq \ell$. Si $\jbf \in \NN^2$ et $m \geq \max(\ell,r(n))$, posons dans $ \hat{L}_\infty^{G_m\dan}$
\[ a_\jbf(c) = \sum_{\kbf \in \NN^2} (-1)^{|\kbf|} (\xbf-\xbf_n)^\kbf \cdot \frac{\partial^{\kbf+\jbf}(c)}{(\kbf+\jbf)!} \binom{\kbf+\jbf}{\jbf}. \]
Les m\^emes arguments que pr\'ec\'edemment montrent que $c = \sum_{\jbf  \in \NN^2} a_\jbf(c) (\xbf-\xbf_n)^\jbf$ dans $\hat{L}_\infty^{G_m\dan}$. Par ailleurs, on a $D_1(a_\jbf(c))=D_2(a_\jbf(c))=0$ et donc aussi $H(a_\jbf(c))=0$ puisque $H=[D_1,D_2]$, ce qui fait qu'il existe $m' \gg 0$ tel que
\[ a_\jbf(c) \in (\hat{L}_\infty^{D_1,D_2,H})^{G_m\dan} = \widehat{K_{m'}(\mu_{p^\infty})}^{G_m\dan} = K_m(\mu_{p^m}), \] 
la derni\`ere \'egalit\'e r\'esultant du th.~\ref{finisan}. Les coefficients $a_\jbf(c_i(y))$ appartiennent donc tous \`a $K_m(\mu_{p^m})$ et dans $\hat{L}_\infty^{G_m\dan}$, on a 
\[ y = \sum_{\jbf \in \NN^2,i \in \NN} a_\jbf(c_i(y)) (\xbf-\xbf_n)^\jbf (z-z_h)^i. \]

Si l'on suppose \`a pr\'esent que $y \in \hat{K}_\infty^{\la}$, alors $\nabla(y) = 0$. On a par ailleurs \[ \nabla(y)=s \cdot \sum_{\jbf \in \NN^2,i \geq 1} i \cdot a_\jbf(c_i(y)) (\xbf-\xbf_n)^\jbf (z-z_h)^{i-1}, \]
ce qui fait que $a_\jbf(c_i(y)) = 0$ si $i \neq 0$. On a aussi $\Tr_{K_\infty(\mu_{p^m})/K_\infty}(y) = [K_\infty(\mu_{p^m}) : K_\infty] \cdot y$ et donc $y = \sum_{\jbf \in \NN^2} y_\jbf (\xbf-\xbf_n)^\jbf$ avec $y_\jbf = [K_\infty(\mu_{p^m}) : K_\infty]^{-1} \cdot \Tr_{K_\infty(\mu_{p^m})/K_\infty}(a_\jbf(c_0(y)))$, qui appartient \`a $K_m$.

L'application $\cup_{n \geq 1} K_{r(n)} \dacc{\xbf-\xbf_n}_n \to \hat{K}_\infty^{\la}$ est une bijection continue entre espaces LB, et donc un isomorphisme d'espaces LB par le th\'eor\`eme de l'image ouverte.
\end{proof}

\section{Structure de $\hat{K}_\infty^{\la}$ dans le cas g\'en\'eral}
\label{drag}
Dans ce chapitre, nous ne faisons pas d'hypoth\`ese sur le groupe de Lie $\Gamma_K$.
Comme $\Gamma_K$ est un groupe de Lie $p$-adique compact, de dimension finie, il existe
(\S 27 de \cite{PSLG}) un groupe analytique ${\mathbb G}$, d\'efini sur $\Qp$, tel que l'on ait
$\Gamma_K={\mathbb G}(\Zp)$.  On note ${\mathfrak g}$ l'alg\`ebre de Lie de
$\Gamma_K$; c'est un $\Zp$-module libre de rang la dimension~$d$ de $\Gamma_K$.

Si $n\geq 1$, on note $\Gamma_n$ le groupe ${\mathbb G}(p^n\Zp)$, image de $p^n{\mathfrak g}$
par l'exponentielle, et on note $K_n$ le sous-corps $K_\infty^{\Gamma_n}$ de
$K_\infty$.  L'anneau $\hat{K}_\infty^{\Gamma_n\dan}$ est une $K_n$-alg\`ebre de Banach;
on note $X_n$ le $K_n$-espace analytique qu'elle d\'efinit.  Nous allons prouver
que $X_n$ devient une boule de dimension~$d-1$ quand on \'etend les scalaires \`a
un corps assez gros.
  Pour \'enoncer le r\'esultat pr\'ecis\'ement, nous allons devoir
introduire certains sous-groupes \`a un param\`etre de ${\mathbb G}$.
Disons que ${\mathfrak a}\in\OO_{\Cp}\otimes_{\Zp}{\mathfrak g}$ est {\it primitif}
si $(\OO_{\Cp}\otimes_{\Zp}{\mathfrak g})/\OO_{\Cp}{\mathfrak a}$ est sans torsion
(et donc est libre, de rang $d-1$, sur $\OO_{\Cp}$).
Si ${\mathfrak a}$ est primitif, on note ${\mathbb H}_{\mathfrak a}$ le sous-groupe \`a un param\`etre
qu'il d\'efinit: si $n\geq1$, alors ${\mathbb H}_{\mathfrak a}(p^n\OO_{\Cp})$ est l'image de
$p^n\OO_{\Cp}$ par l'application $t\mapsto \exp(t{\mathfrak a})$.

\begin{theo}\label{sen1}
Il existe $m\in\NN$ et 
${\mathfrak a}\in\OO_{\hat{K}_\infty(\mu_{p^m})}\otimes_{\Zp}{\mathfrak g}$, 
primitif, tel que,
si $n\geq 1$ et si $L$ est un sous-corps de $\Cp$ contenant 
$\hat{K}_\infty(\mu_{p^m})$, alors
$X_n(L)={\mathbb H}_{\mathfrak a}(p^n\OO_{L})\backslash {\mathbb G}(p^n\OO_L)$.
\end{theo}

\begin{rema}\label{sen2}
(i)  Il r\'esulte de la description ci-dessus que, si $L$ est un sous-corps
complet de $\Cp$ contenant $\hat{K}_\infty(\mu_{p^m})$,
alors $X_n\otimes L$ est une boule de
dimension $d-1$: si ${\mathfrak b}_1,\dots,{\mathfrak b}_{d-1}$ sont tels que
${\mathfrak a},{\mathfrak b}_1,\dots,{\mathfrak b}_{d-1}$ forment une base de
$\OO_{L}\otimes_{\Zp}{\mathfrak g}$ sur $\OO_{L}$, alors
$(x,y_1,\dots,y_{d-1})\mapsto \exp(x{\mathfrak a})\exp(y_1{\mathfrak b}_1)\cdots
\exp(y_{d-1}{\mathfrak b}_{d-1})$ induit un isomorphisme d'espaces analytiques
de $B(0,p^{-n})^d$ sur ${\mathbb G}(p^n\cdot)$,
et donc $(y_1,\dots,y_{d-1})\mapsto \exp(y_1{\mathfrak b}_1)\cdots
\exp(y_{d-1}{\mathfrak b}_{d-1})$ induit un isomorphisme d'espaces analytiques
de $B_{d-1}(0,p^{-n})$ sur $X_n$.

(ii) On en d\'eduit que $\hat{K}_\infty^{\Gamma_n\dan}$ est un anneau de
s\'eries en $d-1$ variables.  
Plus exactement, si $L$ est comme ci-dessus, et si
on note $L\dcroc{X_1,\hdots,X_{d-1}}_{>0}$
l'anneau des germes de fonctions analytiques en $0$ (i.e. des s\'eries de rayon de convergence non nul),
il existe un isomorphisme de $L\dcroc{X_1,\hdots,X_{d-1}}_{>0}$
sur $L\hat\otimes \hat{K}_\infty^{\la}$ envoyant
le sous-anneau des s\'eries de rayon de convergence~$\geq p^{-n}$ sur
$L\hat\otimes \hat{K}_\infty^{\Gamma_n\dan}$.

(iii) On peut d\'emontrer le r\'esultat pr\'ec\'edent \`a la main, si
$K_\infty/K$ est une extension Lubin-Tate (comme au \S\,\ref{calanlt}),
ou si $\Gamma_K$ est un sous-groupe ouvert de $\SL_2(\Zp)$ (comme au \S\,\ref{sl2sec}).
Il suffit d'utiliser les th.~\ref{qpanfinf} et~\ref{kinfansl} en prenant pour variables 
$X_i = 1 \otimes x_i - x_i \otimes 1 \in \Cp \otimes_{K_n} \hat{K}_\infty^{G_n\dan}$, avec $n \gg 0$ 
et $i \in \emb \setminus \{ \Id \}$ dans le cas Lubin-Tate et $i \in \{ 1, 2 \}$ dans le cas $\SL_2(\Zp)$.
\end{rema}

\begin{proof}[D\'emonstration du th.~\ref{sen1}]
Le point de d\'epart de la description de $X_n$
est l'identification 
$$\hat{K}_\infty^{\Gamma_n\dan}={\mathcal C}^\an(\Gamma_n,\hat{K}_\infty)^{\Gamma_n}=
\big(\hat{K}_\infty\hat\otimes_{\Qp} {\mathcal C}^\an(\Gamma_n,\Qp)\big)^{\Gamma_n},$$
l'action de $\Gamma_n$ sur ${\mathcal C}^\an(\Gamma_n,\hat{K}_\infty)$ \'etant donn\'ee
par $(h\cdot\phi)(g)=h\cdot\phi(h^{-1}g)$: dans un sens, on envoie 
$x\in \hat{K}_\infty^{\Gamma_n\dan}$ sur la fonction $\phi_x$ d\'efinie par $\phi_x(g)=g\cdot x$,
dans l'autre sens on envoie $\phi$ sur $\phi(1)$.
L'action de $\Gamma_n$ sur $\hat{K}_\infty^{\Gamma_n\dan}$ correspond,
 via cette identification, \`a l'action $(\gamma,\phi)\mapsto \gamma *\phi$
sur ${\mathcal C}^\an(\Gamma_n,\hat{K}_\infty)$, avec
$(\gamma *\phi)(g)=\phi(g\gamma)$.
Nous allons appliquer la th\'eorie de Sen classique, telle qu'elle
est pr\'esent\'ee dans \cite{PCsen} et~\cite{LB11}, \`a la grosse repr\'esentation
${\mathcal C}^\an(\Gamma_n,\Qp)$.

On fixe un plongement de $\Gamma_K$ dans $\GL_N(\Zp)$, pour un certain $N$.
Si $k\in\NN$, on note $V_k$ l'espace des $\phi:\Gamma_K\to\Qp$ qui sont la restriction
d'un polyn\^ome de degr\'e~$\leq k$ sur $M_N(\Zp)$, et on note ${\mathcal C}^{\rm alg}(\Gamma_K)$
la r\'eunion (croissante) des $V_k$.  (Contrairement \`a ce que la notation
sugg\`ere, cet espace d\'epend en g\'en\'eral 
du plongement de $\Gamma_K$ dans un $\GL_N(\Zp)$.) Comme l'espace des
polyn\^omes de degr\'e~$\leq k$
est stable par $\GL_N(\Zp)$, il l'est a fortiori par $\Gamma_K$,
et chaque $V_k$ peut \^etre vu comme une repr\'esentation de $G_K$ agissant
\`a travers $\Gamma_K$.

Si $n\in\NN$, l'espace ${\mathcal C}^{\rm alg}(\Gamma_K)$ est dense dans
${\mathcal C}^\an(\Gamma_n,\Qp)$ qui est donc son compl\'et\'e pour la norme 
induite.  On note $T_k$ la boule unit\'e de $V_k$ pour cette norme;
c'est l'intersection de $V_k$ avec la boule unit\'e
${\mathcal C}^\an(\Gamma_n,\Qp)^0$ de ${\mathcal C}^\an(\Gamma_n,\Qp)$.
Si $a\geq 1$, alors $\Gamma_{n+a}$ agit trivialement sur ${\mathcal C}^\an(\Gamma_n,\Qp)^0/p^a$
et donc aussi sur $T_k/p^a$.  On choisit $a\geq v_p(12p)$, 
on choisit $m\in\NN$ assez grand (i.e.~$m\geq n(K_{n+a})$, o\`u l'entier
$n(L)$ est celui utilis\'e dans le \S~4.1 de~\cite{LB11}), et on pose
$M=K_{n+a}(\mu_{m})$.
On note $H_M$ et $\Gamma_M$ les groupe $\Gal(\Qpbar/K_\infty(\mu_{p^m}))$
et $\Gal(K_\infty(\mu_{p^m})/M)$.
Alors $H_M\subset H_K$ et $\Gamma_M$ s'identifie \`a un sous-groupe d'indice fini
de $\Gamma_n$ (et m\^eme de $\Gamma_{n+a}$).

Soit $\asen^m$ l'anneau des entiers de $\bsen^m$. On note $\asenM^m$ et $\bsenM^m$
les points fixes de ces anneaux sous l'action de $H_M$.
Comme $T_k$ est fixe par $H_K$, on a $(\asen\otimes_{\Zp}T_k)^{G_M}=(\asenM\otimes_{\Zp}T_k)^{\Gamma_M}$,
et les r\'esultats du \S\,3.3 de \cite{LB11} et le th.~2 de \cite{PCsen} 
impliquent que, pour tout $k$, on a un isomorphisme
\[ \asenM^m \otimes_{\OO_M} (\asenM^m \otimes_{\Zp} T_k)^{\Gamma_M}=\asenM^m \otimes_{\Zp} T_k. \]
L'isomorphisme $\bsenM^m \otimes_M (\bsenM^m \otimes_{\Qp} V_k)^{\Gamma_M} 
= \bsenM^m \otimes_{\Qp} V_k$ est donc une isom\'etrie, et de m\^eme pour 
$\bsenM^m \otimes_M (\bsenM^m \otimes_{\Qp} V_\infty)^{\Gamma_M} = 
\bsenM^m \otimes_{\Qp} V_\infty$, si $V_\infty={\mathcal C}^{\rm alg}(\Gamma_K)$
muni de la norme induite par celle de ${\mathcal C}^\an(\Gamma_n,\Qp)$. Posons $D_\infty = 
(\bsenM^m \otimes_{\Qp} V_\infty)^{\Gamma_M}$. En passant aux compl\'et\'es, 
on trouve que $\bsenM^m \hat{\otimes}_M \hat{D}_\infty = 
\bsenM\hat\otimes_{\Qp} {\mathcal C}^{\an}(\Gamma_n,\Qp)$.
En prenant les points fixes sous l'action de $\Gamma_M$, on en d\'eduit que 
$\hat{D}_\infty = 
\big(\bsenM\hat\otimes_{\Qp} {\mathcal C}^{\an}(\Gamma_n,\Qp)\big)^{\Gamma_M}$, et donc 
que 
$$\bsenM^m \hat{\otimes}_M \big(\bsenM\hat\otimes_{\Qp} {\mathcal C}^{\an}(\Gamma_n,\Qp)\big)^{\Gamma_M}
 = \bsenM\hat\otimes_{\Qp} {\mathcal C}^{\an}(\Gamma_n,\Qp).$$ 

Soit $\theta_g = \frac{d}{du} \otimes 1\otimes 1$ et $\theta_d = 1 \otimes \frac{d}{du}\otimes 1$ 
agissant sur le membre de gauche.
Via l'isomorphisme ci-dessus
$\theta_g + \theta_d$ devient l'op\'erateur $\frac{d}{du}\otimes 1$ agissant
sur le membre de droite, et son noyau est donc
$\hat{K}_\infty(\mu_{p^m})\hat\otimes_{\Qp} {\mathcal C}^{\an}(\Gamma_n,\Qp)=
{\mathcal C}^{\an}(\Gamma_n,\hat{K}_\infty(\mu_{p^m}))$.
Sur ce noyau, $\theta_g$ (\'egal \`a $-\theta_d$) induit une d\'erivation $D$,
et on obtient donc, en prenant l'intersection des noyaux de $\theta_g$ et $\theta_d$,
la relation
$$\hat{K}_\infty(\mu_{p^m})\hat\otimes_M{\mathcal C}^{\an}(\Gamma_n,\hat{K}_\infty(\mu_{p^m}))^{\Gamma_M}={\mathcal C}^{\an}(\Gamma_n,\hat{K}_\infty(\mu_{p^m}))^{D=0}.$$
Enfin, on peut faire une descente \'etale de $M$ \`a $K_n$ et obtenir
$$\hat{K}_\infty(\mu_{p^m})\hat\otimes_{K_n}{\mathcal C}^{\an}(\Gamma_n,\hat{K}_\infty)^{\Gamma_n}={\mathcal C}^{\an}(\Gamma_n,\hat{K}_\infty(\mu_{p^m}))^{D=0}.$$
(Ce r\'esultat est une version de l'identit\'e
$\Cp\otimes_K(\Cp\otimes_{\Qp} V)^{G_K}=
(\Cp\otimes_{K_n}{\rm D}_{{\rm Sen},n}(V))^{\Theta_{\rm Sen}=0}$, cf.~rem.~\ref{rem2}.)

Par ailleurs, on peut faire agir $\Gamma_n$ trivialement sur 
$\bsenM$ et $\hat{K}_\infty(\mu_{p^m})$ et par
$(\gamma *\phi)(g)=\phi(g\gamma)$ sur ${\mathcal C}^{\an}(\Gamma_n,\Qp)$.
Cette action commute \`a celles de $\theta_g$, $\theta_d$ et \`a l'action
pr\'ec\'edente de $\Gamma_n$; elle commute donc aussi \`a $D$.
Autrement dit, $D$ est invariante 
par translation \`a droite.  Elle est donc de la forme 
$\lim_{t\to 0}\frac{1}{t}\big(\phi(e^{t{\mathfrak a}}g)-\phi(g)\big)$, pour un certain 
${\mathfrak a}\in {\hat{K}_\infty}(\mu_{p^m})\otimes_{\Zp}{\mathfrak g}$.

Une comparaison de ce qui se passe pour $n$ et $n+1$ montre que
${\mathfrak a}$ ne d\'epend pas de $n$, et on peut donc faire descendre 
l'isomorphisme ci-dessus \`a $\hat{K}_\infty(\mu_{p^m})$, o\`u $m$ est
ind\'ependant de $n$, et on peut ensuite \'etendre les scalaires \`a tout
sous-corps complet $L$ de $\Cp$ contenant $\hat{K}_\infty(\mu_{p^m})$.
Soit ${\mathbb H}_{\mathfrak a}$ le groupe \`a un param\`etre d\'efini par un multiple primitif de ${\mathfrak a}$.
Le noyau de $D$ s'identifie
aux fonctions analytiques sur ${\mathbb G}(p^n\OO_{L})$,
constantes modulo multiplication \`a gauche par 
${\mathbb H}_{\mathfrak a}(p^n\OO_{L})$,
ce qui permet de conclure.
\end{proof}

\begin{prop}\label{sen3}
L'\'el\'ement ${\mathfrak a}$ de $\Cp\otimes_{\Zp}{\mathfrak g}$ fourni
par le th.~\ref{sen1}
n'est autre que
l'op\'erateur $\Theta_{\rm Sen}$ 
associ\'e \`a la repr\'esentation $V_1$.
\end{prop}
\begin{proof}
Il s'agit d'un exercice de traduction reposant sur le dictionnaire
du \S\,\ref{bsensec} et en particulier la prop.~\ref{classique3}
dont nous reprenons les notations.
Soit $\phi\in V_1$.  On peut d\'ecomposer $\phi$ dans une base 
$\iota(d_1),\dots,\iota(d_N)$ de $\dsen'(V_1)$, o\`u $d_1,\dots,d_N$ est une
base de $\dsen(V_1)$, sous la forme $\phi=\sum_{i=1}^N\alpha_i\iota(d_i)$
avec $\alpha_i\in\bsen^m$.  Alors $\phi$ est tu\'e par $\theta_g+\theta_d$,
et on a $D(\phi)=-\sum_{i=1}^N\alpha_i\frac{d}{du}(\iota(d_i))=
\sum_{i=1}^N\alpha_i e^{-u\Theta_{\rm Sen}}\cdot\Theta_{\rm Sen}(d_i)$.
De plus, comme $\phi\in V_1\subset \Cp\otimes_{\Qp} V_1$, il est constant
vu comme fonction de $u$,
et donc $\phi=\sum_{i=1}^N\alpha_i^{(0)}d_i$.
De m\^eme, $D(\phi)$ est constant comme fonction de $u$ et donc
$D(\phi)=\sum_{i=1}^N\alpha^{(0)}_i\Theta_{\rm Sen}(d_i)$.
Il en r\'esulte que $D=\Theta_{\rm Sen}$ sur $\Cp\otimes_{\Qp} V_1$.
Le r\'esultat s'en d\'eduit.
\end{proof}

\begin{rema}\label{sen4}
{\rm (i)} Comme $D\neq 0$ puisqu'il est induit 
par $1\otimes\frac{d}{du}$ sur $\bsen\otimes\bsen$, on en d\'eduit
que $\Theta_{\rm Sen}\neq 0$.

{\rm (ii)}  
On a 
$${\mathcal C}^\an(X_n(\Cp),\Cp)=\Cp\hat\otimes_{K_n} {\mathcal C}^{\an}(\Gamma_n,\Cp)^{\Gamma_n}=
{\mathcal C}^{\rm an}({\mathbb G}(p^n\OO_{\Cp}),\Cp)^{D=0}.$$
Si $\sigma\in G_{K_n}$, l'action de $\sigma\otimes 1$ au milieu
devient l'action standard 
$(\sigma\cdot\phi)(x)=\sigma(\phi(\sigma^{-1}(x)))$ \`a gauche,
mais \`a droite cette action standard est tordue par l'action
de $\Gamma_n$, et est donn\'ee par la formule
$(\sigma\star\phi)(x)=\sigma\big(\phi(\sigma^{-1}(\gamma(\sigma)^{-1}x))\big)$, o\`u
l'on a not\'e $\gamma:G_K\to\Gamma_K$ l'application naturelle (cette
formule est celle que l'on a utilis\'ee pour $x\in\Gamma_n={\mathbb G}(p^n\Zp)$,
auquel cas $\sigma^{-1}(\gamma(\sigma)^{-1}x)=\gamma(\sigma)^{-1}x$;
 elle est donc vraie pour tout $x$ par prolongement
analytique).  On en d\'eduit, en notant $\pi:{\mathbb G}(p^n\OO_{\Cp})\to X_n(\Cp)$
l'application fournie par le th.~\ref{sen1}, que l'on a
$$\sigma(\pi(x))=\pi(\gamma(\sigma)\sigma(x)),\quad{\text{si $x\in {\mathbb G}(p^n\OO_{\Cp})$
et $\sigma\in G_{K_n}$.}}$$

{\rm (iii)} La formule ci-dessus passe au quotient par ${\mathbb H}_{\mathfrak a}$ car
$D=\Theta_{\rm Sen}$ et $\Theta_{\rm Sen}$ commute \`a l'action de $G_K$
sur $\Cp\otimes_{\Zp}{\mathfrak g}\subset \Cp\otimes_{\Qp} {\rm End}(V_1)$,
ce qui se traduit par $\gamma(\sigma)\sigma({\mathfrak a})\gamma(\sigma)^{-1}={\mathfrak a}$,
pour tout $\sigma\in G_{K_n}$ (avec $\sigma=\sigma\otimes 1$ sur
$\Cp\otimes_{\Qp} {\rm End}(V_1)$).
\end{rema}

\providecommand{\bysame}{\leavevmode ---\ }
\providecommand{\og}{``}
\providecommand{\fg}{''}
\providecommand{\smfandname}{\&}
\providecommand{\smfedsname}{\'eds.}
\providecommand{\smfedname}{\'ed.}
\providecommand{\smfmastersthesisname}{M\'emoire}
\providecommand{\smfphdthesisname}{Th\`ese}

\end{document}